\documentclass[11point]{amsart}
\usepackage{color}

\usepackage{amssymb}
{

\usepackage{geometry}                
\geometry{letterpaper}                   

\newtheorem{theorem}{Theorem}[section]
\newtheorem{proposition}[theorem]{Proposition}
\newtheorem{lemma}[theorem]{Lemma}
\newtheorem{corollary}[theorem]{Corollary}
\theoremstyle{definition}
\newtheorem{definition}[theorem]{Definition}

\theoremstyle{remark}
\newtheorem{remark}[theorem]{Remark}

\numberwithin{equation}{section}
\newcommand{\be}{\begin{equation}}
\newcommand{\ee}{\end{equation}}

\newcommand{\bbC}{{\mathbb C}}

\newcommand{\bbR}{{\mathbb R}}

\newcommand{\calU}{{\mathcal U}}

\newcommand{\calL}{{\mathcal L}}

\newcommand{\calP}{{\mathcal P}}
\newcommand{\calH}{{\mathcal H}}

\newcommand{\calS}{{\mathcal S}}

\newcommand{\calN}{{\mathcal N}}

\newcommand{\frakh}{{\mathfrak h}}

\newcommand{\frakp}{{\mathfrak p}}

\newcommand{\inner}[2]{\langle#1,#2\rangle}

\newcommand{\norm}[1]{\lVert#1\rVert}

\newcommand{\PB}[2]{\{  #1\,,\,#2 \}}

\renewcommand{\Box}{\square}

\newcommand{\wt}{\widetilde}

\newcommand{\h}{\hbar}
\newcommand{\hmtwo}{\hbar^{-1/2}}

\newcommand{\hinv}{\hbar^{-1}}

\setcounter{tocdepth}{4}

\begin{document}

\title{Semiclassical states associated to isotropic submanifolds of phase space}
\author{V. Guillemin}
\address{Mathematics Department\\
Massachusetts Institute of Technology\\Cambridge, MA 02138}
\email{vwg@math.mit.edu}
\author{A. Uribe}
\address{Mathematics Department\\
University of Michigan\\Ann Arbor, MI 48109}
\email{uribe@umich.edu}
\author{Z. Wang}
\address{School of Mathematical Sciences\\
University of Science and Technology of China\\
Hefei, Anhui 230026, P. R. China.}
\email{wangzuoq@ustc.edu.cn}
\thanks{Z. W. supported in part by  NSF in China 11571331 and 11526212.}

\begin{abstract}
We define classes of quantum states associated to isotropic submanifolds
of cotangent bundles.  The classes are stable under the action of semiclassical
pseudo-differential operators and covariant under the action of semiclassical
Fourier integral operators.  We develop a  
symbol calculus for them; the symbols are symplectic
spinors.  We outline various applications. 
\end{abstract}

\maketitle

\centerline{ Dedicated to Louis Boutet de Monvel}

\setcounter{tocdepth}{2}
\tableofcontents

\section{Introduction}

Among the contributions to mathematics that Boutet de Monvel is 
most remembered for is his work on Hermite distributions and Toeplitz 
operators,  and the purpose of this paper is to give a semi-classical 
account of this theory, an account that is largely inspired by Boutet's 
paper \cite{Bout}  on partial differential equations with multiple 
characteristics (in which he introduces Hermite distributions for the 
first time), his paper \cite{BS} with Sj\"ostrand (in which  the classical 
theory of the Bergman projector for bounded domains in $\bbC^n$ is shown to 
have an elegant and succinct microlocal description as a Toeplitz 
operator), and the monograph, \cite{BG}, (co-written with one of the authors 
of this paper) in which the theory of Toeplitz operators is developed ab 
ovo and the spectral properties of these operators are investigated in 
detail.

In what follows $M$ will be a smooth manifold and $T^*M$ its cotangent bundle. We recall that if one is given a Lagrangian submanifold, $\Lambda$, of $T^*M$ on can associate to $\Lambda$ a space, $I(M, \Lambda)$, of oscillatory functions having the property that their semi-classical wave front sets are contained in $\Lambda$; and one can define for these functions a symbol calculus which has nice functorial properties with respect to composition by semi-classical pseudodifferential operators. (For more details, see, for instance chapter 8 of the reference \cite{GSSemiclassical}.) Our goal below will be to develop an analogue of this theory for isotropic submanifold of $T^*M$ similar to the isotropic analogue of the classical theory of Lagrangian distributions: the theory of Hermite distributions developed by Boutet and his collaborators in the references cited above. (The salient property of Hermite distribution is that their microsupports are contained in an isotropic submanifold, $\Sigma$, of $T^*M$, and this will be the case for our ``semi-classical oscillatory functions of isotropic type" as well.)

We will begin in \S 2 by defining these functions in the special case for which $\Sigma$ is a vector subspace of the zero section in $T^*\bbR^n$ and show that in this ``model" case these functions have nice compositional properties with respect to semi-classical pseudodifferential and Fourier integral operators, have a well-defined symbol calculus and satisfy analogues of the first and second order transport equations described in \cite{Gui}, \S 10. Then in \S 3, we will extend this theory to manifold setting; and, in particular, show how to describe intrinsically the manifold versions of the results in \S 2. 

Finally in \S 4 we will briefly discuss some applications of this theory to partial differential equations and several complex variable theory (applications which we intend to discuss in more detail in a projected sequel to this paper.)

\section{Local theory}

\subsection{The model class}

Let us fix a splitting $\bbR^n = \bbR^k\times\bbR^l$ of Euclidean space, together
with a splitting of the standard coordinates $x = (t, u)$.  The dual coordinates in
$T^*\bbR^n = \bbR^{2n}$ will be denoted $\xi = (\tau, \mu)$.  We denote:
\begin{equation}\label{}
Y = \bbR^k\times\{0\}\subset\bbR^n, \quad \Sigma_0 = \{(t, u=0; \xi=0)\} \subset T^*\mathbb R^n.
\end{equation}
$\Sigma_0$ is an isotropic submanifold of $T^*\bbR^n$, as it is contained in the zero section.

\begin{definition}
Let $r$ be a half-integer.  The class $I^r(\Sigma_0)$ consists of all smooth,
$\h$-dependent functions of the form
\begin{equation}\label{basicdef}
\Upsilon(t,u,\h) = a(t,\hmtwo u,\h): \bbR^n\times (0,\h_0)\to\bbC,
\end{equation}
where $a(t,u,\h)$ has an asymptotic expansion as $\h\to 0$ of the form
\begin{equation}\label{2exp}
a(t,u,\h) \sim \h^{r}\sum_{j=0}^\infty a_j(t,u)\, \h^{j/2},
\end{equation}
where, $\forall j$, $a_j(t,u)$ is 
a Schwartz function in the $u$ variable with estimates locally uniform in $t$.  
More precisely, we assume that for all $j$
\begin{equation}\label{}
\forall   \alpha, \, \beta,\, N, \forall K\subset\bbR^k\ \text{compact}, \ 
\ \exists C\ \text{such that} \ \forall (t,u)\in K\times\bbR^l,
\ |\partial^\alpha_t\partial^\beta_u a_j(t,u)| \leq C \left(1+\norm{u}\right)^{-N},
\end{equation}
and the expansion (\ref{2exp})
means that $\forall \alpha, \, \beta,\, N, \forall K\subset\bbR^k\ \text{compact},\ 
\ \exists C\ \text{such that} \ \forall (t,u)\in K\times\bbR^l,$
\begin{equation}\label{}
\left|\partial^\alpha_t\partial^\beta_u \left(a(t,u)-\sum_{j=0}^Nh^{r+j/2}a_j(t,u)\right)\right| \leq C \hbar^{r+N+1}.
\end{equation}
\end{definition}

The following is not hard to prove:
\begin{lemma}
The semiclassical wave-front set of $a\in I^r(\Sigma_0)$ is contained in $\Sigma_0$.
\end{lemma}

\begin{definition}
Let $\Upsilon\in I^r(\Sigma_0)$ is as in the previous definition.  Then its {\em {rough} symbol}
is the function 
\begin{equation}\label{}
\begin{array}{rcc}
\sigma_\Upsilon:\Sigma_0 & \to & \calS(\bbR^l)\\
(t,0;0) & \mapsto & a_0(t, \cdot).
\end{array}
\end{equation}
where $\calS(\bbR^l)$ denotes the class of Schwartz functions on $\bbR^l$.
\end{definition}

\begin{remark}
We make the following remarks about the symbol, anticipating the general definition:
At every point $s\in\Sigma_0$ the symplectic normal space
\begin{equation}\label{}
\calN_s := \left(T_s\Sigma_0\right)^\circ/T_s\Sigma_0
\end{equation}
can be canonically identified with the $(u,\mu)$ symplectic vector space.
This vector space inherits a polarization $L_s\subset\calN_s$ (a Lagrangian subspace)
from the vertical polarization of $T^*\bbR^n$, namely $L_s = \{u=0\}$.
At every point $s\in \Sigma_0$, $\sigma_\Upsilon(s)$ is to be thought of as a Schwartz function in the quotient
$\calN_s/L_s\cong\bbR^l$.  {To obtain an invariant
version of the symbol, in the manifold case, we will have to 
``decorate" the rough symbol with half-forms, 
so that the invariant symbol will be an object of the form
\begin{equation}\label{}
(t, 0; 0)\mapsto a_0(t, \cdot)\, (\wedge^{1/2} dt)\,
(\wedge^{1/2}du),
\end{equation}
where we henceforth let
\begin{equation}\label{}
\wedge^{1/2} dt = (dt_1\wedge\cdots\wedge dt_k)^{1/2}\quad \text{and}
\quad \wedge^{1/2}du =(du_1\wedge\cdots\wedge du_l)^{1/2}.
\end{equation}
}
\end{remark}

\subsection{Invariance under $\Psi$DOs and the local transport equations}

It is clear from the definition that the class $I^r(\Sigma_0)$ is invariant under the
action of differential operators that are supported near $Y$ (or whose coefficients have
at most polynomial growth in $u$).  We now prove:

\begin{theorem}\label{psiDOInv}
Let $\Upsilon\in I^r(\Sigma_0)$ and $P$ a
semiclassical pseudodifferential operator of degree $d$
on $\bbR^n$ whose Schwartz kernel
is compactly supported in the $u$ variables. Then 
\begin{equation}\label{}
P(\Upsilon)\in I^{r+d}(\Sigma_0)\quad\text{and}\quad \sigma_{P(\Upsilon)}(s) = \frakp(s)\sigma_\Upsilon(s)
\end{equation}
where $\frakp$ is the symbol of $P$.
\end{theorem}
\begin{proof}
For simplicity and for the purpose of proving this theorem as well as proving the next two theorems, we write 
\[\Upsilon(t,u,\hbar)  = \hbar^ra(t,\frac u{\sqrt \hbar}, \hbar)=\hbar^r a_0(t,\frac{u}{\sqrt \h})+\hbar^{r+\frac 12}a_1(t, \frac{u}{\sqrt \h})+\hbar^{r+1}a_2(t,\frac{u}{\sqrt \h})\]
and we assume that the symbol of $P$ is 
\[
\hbar^d \frakp(t,u,\tau,\mu) = \hbar^d p_0(t,u,\tau,\mu) + \hbar^{d+1} p_1(t,u,\tau,\mu).
\]
Then 
\[\aligned
P(\Upsilon) (t, u) &= \frac 1{(2\pi\hbar)^n} \iint e^{\frac i\h [(t-\tilde t)\cdot \tau+(y-\tilde u) \cdot \mu]} \hbar^d p(t,u,\tau,\mu, \hbar)\hbar^r a\left(\tilde t, \frac{\tilde u}{\sqrt \h}, \hbar \right)d\tilde t d\tilde u d\tau d\mu  \\
& = \frac {\h^{d+r}}{(2\pi\hbar)^{n}}\int e^{\frac i\hbar (t \cdot \tau + u \cdot \mu)} p(t, u,\tau, \mu, \hbar) \hbar^{\frac l2} \hat a(\frac{\tau}{\h}, \frac{\mu}{\sqrt \h}, \hbar) d\tau  d\mu \\
& = \frac{\hbar^{d+r}}{(2\pi\hbar)^n} \hbar^{\frac l2}\hbar^{k+\frac l2} \int e^{i(t \cdot \tau+\frac{u}{\sqrt\hbar} \cdot \mu)}p(t,u,\hbar \tau, \sqrt \hbar \mu, \h)  \hat a(\tau, \mu, \h) d\tau d\mu \\
& = \hbar^{d+r} b(t,\frac u{\sqrt h}, \hbar),
\endaligned\]
where $\hat a$ is the Fourier transform of $a$, and $b$ is the function 
\begin{equation}\label{PUpsilon}
b(t,u, \hbar) = \frac 1{(2\pi)^n} \int e^{i(t \cdot \tau+u \cdot \mu)}p(t,\sqrt \h u,\hbar \tau, \sqrt \hbar \mu,\hbar) \hat a(\tau, \mu, \hbar) d\tau d\mu.
\end{equation}
In particular, $b(t,u, 0) = p_0(t,0,0,0)a_0(t,u).$
The conclusion follows.
\end{proof}

If $\frakp|_{\Sigma_0} \equiv 0$ then $P(\Upsilon)\in I^{r+d+1/2}(\Sigma_0)$, and one can ask what its symbol is.  
\begin{theorem}\label{localTransportI}
In the situation of Theorem \ref{psiDOInv}, assume that $\frakp|_{\Sigma_0} \equiv 0$.
Then $P(\Upsilon)\in I^{r+d+1/2}(\Sigma_0)$ and
\begin{equation}\label{localTransport1}
\sigma_{P(\Upsilon)}(t,0;0)(u) = 
\left(\sum_{j=1}^l u_j\frac{\partial\frakp}{\partial u_j}(t,0;0)\right)a_0(t,u) + \frac 1i\sum_{j=1}^l\frac{\partial\frakp}{\partial \mu_j}(t,0;0)\frac{\partial a_0}{\partial u_j}(t,u).
\end{equation}
\end{theorem}
\begin{proof}
If $p_0(t,0,0,0)=0$, then  
\[
p(t, \sqrt \hbar u, \hbar \tau, \sqrt \hbar \mu, \hbar) = \hbar^{1/2}  \sum_{j=1}^l\left(  u_j \frac{\partial p_0}{\partial u_j} + \mu_j \frac{\partial p_0}{\partial \mu_j} \right)  + O(\hbar). 
\]
Therefore (\ref{PUpsilon}) becomes 
\[\aligned 
b(t,u, \hbar)  & =\hbar^{1/2}  \int e^{i(t \cdot \tau+u \cdot \mu)}
\sum_{j=1}^l\left(  u_j \frac{\partial p_0}{\partial u_j} + \mu_j \frac{\partial p_0}{\partial \mu_j}  \right)  
\hat a_0(\tau, \mu) d\tau d\mu+O(\hbar)
\\& 
=\hbar^{1/2} \int e^{i(t \cdot \tau+u \cdot \mu)}
\sum_{j=1}^l\left(  u_j \frac{\partial p_0}{\partial u_j} \hat a_0(\tau, \mu) + \frac 1i  \frac{\partial p_0}{\partial \mu_j} 
\widehat{\frac{\partial a_0}{\partial u_j}}(\tau, \mu)  \right)  d\tau d\mu+O(\hbar)
\endaligned\]
where we have used that 
$\mu_j\hat a_0(\tau, \mu) = \frac 1i \widehat{\frac{\partial a_0}{\partial u_j}}(\tau, \mu)$.
The functions $\frac{\partial p_0}{\partial u_j},   \frac{\partial p_0}{\partial \mu_j} $ are all evaluated 
at the point $(t,0,0,0)$.  The conclusion follows.
\end{proof}

\begin{remark}
The right-hand side of (\ref{localTransport1}) has a nice interpretation that carries over to the
general case.  
Since $\Sigma_0$ is isotropic, $\frakp|_{\Sigma_0}\equiv 0$
implies that the Hamilton vector field of $\frakp$ at $s\in \Sigma_0$, $\Xi_\frakp(s)$, 
belongs to the annihilator $\left(T_s\Sigma_0\right)^\circ$.  It projects to 
the symplectic normal space, $\calN_s$, and therefore defines 
a Lie algebra element of the Heisenberg group of $\calN_s$.
 (\ref{localTransport1}) is just the action of this Lie algebra element on $\sigma_\Upsilon(s)$.
\end{remark}

One can go further down still, under the following assumption:
\[
\tag{$\ast$}  
\frakp|_{\Sigma_0} \equiv 0\quad\text{and}\quad \Xi_\frakp\ \text{is tangent to}\ \Sigma_0.
\]

\begin{theorem}\label{localTransportII}
In the situation of Theorem \ref{psiDOInv}, assume further that ($\ast$) holds.
Then $P(\Upsilon)\in I^{r+d+1}(\Sigma_0)$ and its symbol is
\begin{equation}\label{localTransport2}
\aligned
\sigma_{P(\Upsilon)}(s)(u) = &
  p_1(s) a_0(t, u)+\frac 1{\sqrt{-1}} \sum  \frac{\partial p_0}{\partial \tau_j}(s) \frac{\partial a_0}{\partial t_j}(t,u) 
\\
+ \frac 12 \sum \left(   \frac{\partial^2p_0}{\partial u_i\partial u_j}(s)  a_0(t, u) u_iu_j\right.
&\left.+\frac 2{\sqrt{-1}}  \frac{\partial^2p_0}{\partial u_i\partial \mu_j}(s)\frac{\partial a_0}{\partial u_j}(t,u) u_i  -      \frac{\partial^2p_0}{\partial \mu_i\partial \mu_j}(s) \frac{\partial^2 a_0}{\partial u_i\partial u_j} (t,u) \right)
\endaligned
\end{equation}

\end{theorem}
\begin{proof}
The fact that $\Xi_{p_0}$ is tangent to $\Sigma_0$ implies that at $s=(t,0,0,0)$
\[ 
\frac{\partial p_0}{\partial u_j}(s)=\frac{\partial p_0}{\partial \mu_j}(s)=0, \quad 1 \le j \le l. 
\]
So modulo terms of $o(\hbar)$, one has 
\[\aligned 
p(t,\sqrt \h u, \h \tau, \sqrt \h \mu) = \hbar & p_1(s)   +
\hbar \sum \frac{\partial p_0}{\partial \tau_j}(s)\tau_j 
\\
& + 
\frac \h2 \sum \left( \frac{\partial^2p_0}{\partial u_i\partial u_j}(s)u_iu_j+ 
  \frac{\partial^2p_0}{\partial \mu_i\partial \mu_j}(s)\mu_i\mu_j+ 
 2 \frac{\partial^2p_0}{\partial u_i\partial \mu_j}(s)u_i\mu_j \right)
\endaligned
\]
Substituting this into (\ref{PUpsilon}), 
  the conclusion follows. 
\end{proof} 
\begin{remark}
Once again, the right-hand side of (\ref{localTransport2}) has an 
interesting interpretation, 
 {see Theorem
\ref{symbolcalc2}.  For now we simply point out that the second
line in (\ref{localTransport2}) is the action of the infinitesimal 
metaplectic representation of the hessian of $p_0$ with respect
to the $(u,\mu)$ variables, on the function $a_0(t,\cdot)$.}
\end{remark}

\subsection{Invariance under FIOs preserving $\Sigma_0$}

In this section we prove that the model classes $I^{\bullet}(\Sigma_0)$ are invariant 
under the action of zeroth-order
semiclassical Fourier integral operators associated to (not necessarily homogeneous)
canonical transformations
$f: T^* \bbR^n \to T^* \bbR^n $ that preserve $\Sigma_0$ as a set.  By transplantation, 
this will allow us to define isotropic states on manifolds.  In the next 
section we will study 
how the ({rough})
symbols transform under the action of FIOs.

We state our main theorem: 
\begin{theorem}\label{MLT2}
If $\gamma: T^*\bbR^n \to T^*\bbR^n$ is a symplectomorphism mapping $\Sigma_0$ into $\Sigma_0$ and $F_\gamma$ a semiclassical Fourier integral operator quantizing $\gamma$, then $F$ maps $I(\Sigma_0)$ into $I(\Sigma_0)$. 
\end{theorem}

Note that, since we already know that the classes $I^{\bullet}(\Sigma_0)$ are invariant under $\Psi$DOs, without loss of generality
we will only consider FIOs preserving $\Sigma_0$ and having an amplitude identically
equal to one.

To prove theorem \ref{MLT2}, we shall first study a number of special cases, and then we will see that the general case is a combination of these special cases.   

\subsubsection{Invariance under FIOs associated with lifted diffeomorphisms}
Let $f: \bbR^n \to \bbR^n$ be a diffeomorphism which maps  $Y$ onto itself. Then $f$ lifts to a symplectomorphism 
\begin{equation}\label{ML9}
\gamma_f: (x, \xi) \mapsto (y, \eta), \quad y=f(x), \xi=df_x^*\eta
\end{equation}
which maps $\Sigma_0$ onto itself and $\gamma_f$ is quantized by the pull-back map, 
\[
f^*: C^\infty(\bbR^n) \to C^\infty(\bbR^n).
\]
We first observe
\begin{lemma}\label{FIOInv0}
The pull-back operator $f^*$ maps $I^r(\Sigma_0)$ onto itself.
\end{lemma}
\begin{proof}
If we write $f(t, u) = (\tilde t, \tilde u) =  (f_1(t, u), f_2(t, u))$, then the condition $f(Y) \subset Y$ implies $f_2(t, u)=g_2(t, u)u$. As a consequence,
\[
f^*\Upsilon(t,u, \hbar) = a\left(f_1(t, u), \frac{f_2(t, u)}{\sqrt \hbar}, \hbar \right) = a\left(f_1(t, u), g_2(t,u)\frac{u}{\sqrt \hbar}, \hbar\right) 
\] 
is an element in $I^r(\Sigma_0)$. 
\end{proof}
  
\subsubsection{Invariance under FIOs associated with symplectomorphisms fixing the zero section}

Next we will consider semiclassical FIOs for which  the underlying  canonical transformation $\gamma: T^*\bbR^n \to T^*\bbR^n$ maps the zero section onto itself identically.  Since the zero section is carried out to $T_0^*\bbR^n$ by the symplectomorphism $(x, \xi) \mapsto (-\xi, x)$, it suffices to characterize all symplectomorphisms $\gamma$ which carry the  zero section in $T^*\bbR^n$ identically to $T_0^*  \bbR^n$. 
\begin{lemma}\label{gfHtV}
If $\gamma: T^*\mathbb R^n \to T^*\bbR^n$ is a symplectomorphism mapping the cotangent space $T_0^*\bbR^n=\bbR^n$ identically onto the zero section, $\bbR^n$ in $T^*\bbR^n$, then it has to be horizontal, i.e. defined by a generating function, $\varphi(x, y) \in C^\infty(\bbR^n \times \bbR^n)$, with the property
\[
\gamma(x,\xi)=(y, \eta)  \mbox{\ iff\ } \xi=-\frac{\partial \phi}{\partial x}, \; \eta=\frac{\partial \phi}{\partial y}.
\]
Moreover, $\varphi$ has to be of the form 
\[
\varphi(x, y, \hbar)=x\cdot y+\sum \varphi_{i\!,\!j}(x, y, \hbar)y_iy_j
\] 
and the Fourier integral operator quantizing $\gamma$ has to have a Schwartz kernel of the form 
\begin{equation}\label{ML1}
e^{\frac{i}{\hbar}\left(x \cdot y+\sum \varphi_{i\!,\!j}(x,y,\hbar)y_iy_j \right)}a(x,y,\hbar).
\end{equation}
\end{lemma}
\begin{proof}
``Horizontality" is equivalent to the condition that the graph, $\Gamma$, of $\gamma$ is nowhere vertical, or alternatively, that for any pair $p, q \in M=T^*\bbR^n$ with $q=\gamma(p)$, $d\gamma_p: T_pM \to T_qM$ does not map vectors $v \ne 0$ tangent to the cotangent fiber of $M$ at $p$ onto vectors, $w=d\gamma_p(v)$ tangent to the cotangent fiber of $M$ at $q$. However in a neighborhood of the zero section in $M$ this is ruled out by the condition  that $\gamma$ maps the zero section $\bbR^n$  of $M$ onto $T_0^* \bbR^n$.

Now assume that $\Gamma$ is horizontal, with generating function $\phi$.  Let $\gamma$ be the corresponding symplectomorphism.   We would like to find the conditions so that $\gamma$ maps the zero section onto $T_0^*\mathbb R^n$. Clearly this is the case if and only if for all $x \in \mathbb R^n$, 
\begin{equation}
\frac{\partial}{\partial x}\phi(x, 0)=0,
\end{equation}
or equivalently if $\phi(x, 0)$ is a constant function of $x$, and without loss of generality we can assume that this constant function is zero, i.e. 
\begin{equation}
\phi(x, y) = \sum y_i \phi_i(x)+\sum y_i y_j \phi_{i, j}(x, y).
\end{equation}
Moreover, for $\phi(x, y)$ to be a generating function of a symplectomorphism we need to require that the matrix 
\begin{equation}
\left[
\frac{\partial^2 \phi}{\partial x_i\partial y_j}(x, y)
\right]
\end{equation}
is everywhere of rank $n$, and hence in particular that $[\frac{\partial \phi_i}{\partial x_j}(x)]$ is invertible. Thus modulo a change of variables (which, as we have seen, will not change the class $I(\Sigma)$) we can assume 
\begin{equation}
\phi(x, y) = \sum x_i y_i + \sum y_iy_j\phi_{i, j}(x, y, \hbar). 
\end{equation}
Hence   the F.I.O. quantizing $\gamma$ is the operator 
\begin{equation}
A_\gamma f(y) = \int e^{\frac i\hbar (x\cdot y + \sum y_iy_j \phi_{ij}(x,y,\hbar))} a(x,y,\hbar)f(x) dx. 
\end{equation}
\end{proof}
 
To obtain the general transformation preserving the zero section, it suffices to pre-compose 
the $\gamma$'s alluded above with the transformation $J:(x,\xi)\mapsto (-\xi,x)$, which is 
associated to the semiclassical Fourier transform.  So we need 

\begin{lemma}
Let $\Sigma_1 = \{(x=0; \tau,\mu=0)\}\subset T_0^*\bbR^n$
and let $I^r(\Sigma_1)$ denote the image of $I^r(\Sigma_0)$ under the semiclassical Fourier
transform.  Then the elements of $I^r(\Sigma_1)$ are precisely the functions of the form
\begin{equation}\label{}
b\left(\hinv t, \hmtwo u, \h \right)
\end{equation}
where $b$ has an asymptotic expansion as before. 
\end{lemma}
\begin{proof}
Let $\Upsilon(t,u,\hbar)=a(t, \hbar^{-1/2}u, \hbar)$ be an element in $I^r(\Sigma_0)$. Then
its semi-classical Fourier transform is   
\[
\mathcal F_\hbar \Upsilon(\tau, \mu) = \int e^{-\frac i\hbar (t \cdot \tau+u \cdot \mu)}a(t, \frac u{\sqrt \hbar}, \hbar) dtdu 
= \int e^{-i(t \cdot \frac \tau \hbar + \frac{u}{\sqrt \hbar} \cdot \frac {\mu}{\sqrt \hbar} )}a(t, \frac u{\sqrt \hbar}, \hbar) dtdu ,
\]
which can be written as the form $b\left(\hinv \tau, \hmtwo \mu, \h \right)$, where
\[
b\left(\tau, \mu, \h \right) = \int e^{-i(t\cdot \tau+\frac u{\sqrt\hbar} \cdot \mu)}a(t, \frac u{\sqrt \hbar}, \hbar) dtdu
\]
is the Fourier transform of $a$, and thus has an asymptotic expansion as in (\ref{2exp}).  
\end{proof}
 
Now suppose $b(\frac{t}{\hbar}, \frac{u}{\sqrt \hbar}, \hbar)$ is an isotropic function on $\mathbb R^n$ with microsupport on the subset $\mu=0$ of $T_0^*\bbR^n$ and framed by $T_0^*\bbR^n$.  Let $(Fb)(t, u, \hbar)$ be equal to 
\begin{equation}\label{ML2}
\int e^{\frac i\hbar \left(t \cdot \tilde t+u \cdot \tilde u + 
\sum \tilde t_i\tilde t_j  \varphi_{ij}(t,u,\tilde t,\tilde u, \hbar)+
\sum \tilde u_i \tilde u_j \phi_{ij}(t,u,\tilde t, \tilde u,\hbar)+
\sum \tilde t_i \tilde u_j \psi_{ij}(t,u,\tilde t, \tilde u,\hbar) \right)} 
a(t,u,\tilde t, \tilde u, \hbar)b(\frac{\tilde t}{\hbar}, \frac{\tilde u}{\sqrt \hbar}, \hbar)d\tilde t d \tilde u.
\end{equation}
Replacing in this integral $\tilde t_i$ by $\hbar \tilde t_i$ 
and $\tilde u_i$ by $\sqrt \hbar \tilde u_i$, 
it becomes $\hbar^{k+\frac l2}g(t, \frac{u}{\sqrt \hbar}, \hbar)$, where 
\begin{equation}\label{ML3}
g(t, u, \hbar) = \int e^{i\left( (t,u)\cdot(\tilde t, \tilde u)+\sum \tilde u_i \tilde u_j \tilde \phi_{ij}+
\hbar \sum \tilde t_i \tilde t_j  \tilde \varphi_{ij} +
\sqrt\hbar \sum \tilde t_i \tilde u_j \tilde \psi_{ij} \right)}a(t,u,\hbar \tilde t, \tilde u, \hbar) 
b(\tilde t,\tilde u, \hbar) d\tilde t d \tilde u,
\end{equation}
with 
\[
\tilde \varphi_{ij}(t, u, \tilde t, \tilde u, \hbar)=\varphi_{ij}(t, \sqrt \hbar u, \hbar \tilde t, \sqrt \hbar \tilde u, \h)
\]
and likewise for $\tilde \phi_{ij}$ and $\tilde \psi_{ij}$. 
Note that in particular 
\begin{equation}\label{ML4}
g(t, u, 0) = a(t,u,0,0,0)\int e^{i((t,u)\cdot (\tilde t, \tilde u)+\sum \tilde u_i \tilde u_j \phi_{ij}(t, 0))} b(\tilde t,\tilde u, 0)d\tilde t d\tilde u. 
\end{equation}

Let's next compose the operator (\ref{ML2}) with the semiclassical Fourier transform 
\begin{equation}
\Upsilon = b(\tilde t, \frac{\tilde u}{\sqrt \hbar}) \mapsto \hat b(\frac{\tilde t}{\hbar}, \frac{\tilde u}{\sqrt \hbar}).
\end{equation}
This gives us an operator 
\begin{equation}\label{ML6}
\Upsilon \mapsto \int e^{\frac i\hbar \left( (t,u)\cdot (\tilde t,\tilde u)+\sum \tilde t_i \tilde t_j  \varphi_{ij} +\sum \tilde u_i \tilde u_j \phi_{ij} + \sum \tilde t_i \tilde u_j \psi_{ij} \right) }  a(t,u,\tilde t, \tilde u, \hbar)\hat b(\frac{\tilde t}{\hbar}, \frac{\tilde u}{\sqrt \hbar}) d\tilde t d\tilde u
\end{equation}
and by the lemma \ref{gfHtV} above this operator is the quantization of a symplectomorphism $\gamma: T^*\bbR^n \to T^*\bbR^n$ mapping the zero section identically onto itself. Moreover, the operator (\ref{ML6}) maps the space of functions $I^\bullet(\Sigma_0)$ onto itself when $\Sigma_0$ is the subset $\tilde u=0$ of the zero section of $T^*\bbR^n$ (by the calculation we've just made). Also note that if we set $\hbar=0$ we get by the formulae (\ref{ML3}) and (\ref{ML4}) the expression 
\begin{equation}\label{ML7}
\hbar^{k+\frac l2} g(t, u, 0) = a(t,u,0,0,0)\int e^{i (t,u)\cdot (\tilde t,\tilde u)}e^{i \sum \tilde u_i \tilde u_j \phi_{ij}(t, 0)} \hat b(\tilde t,\tilde u, 0)d\tilde t d\tilde u
\end{equation}
for the leading term in (\ref{ML6}). In other words for $t$ fixed the function 
\[
\sigma_{t}(g)(u):=g(t, u, 0)
\]
is the function 
\begin{equation}\label{ML8}
\sigma_t(a)(F_2^{-1} e^{i\sum \lambda_{ij}^t u_i u_j} F_2) \sigma_{t}(b), \quad \sigma_{t}(f) = f(t, u, 0),
\end{equation}
where $F_2$ is the Fourier transform $h(u) \mapsto \int e^{-iu\tilde u}h(u)du$ and $\lambda_{ij}^{t}$ is the constant $\phi_{ij}(t, 0)$. 
 
\subsubsection{Invariance under FIOs associated to fiber-preserving symplectomorphisms}

 Let $\phi: \bbR^n \to \bbR$ be a $C^\infty$ function with the property 
\begin{equation}\label{ML10}
\phi(t, 0)=\frac{\partial \phi}{\partial x}(t, 0) = 0.
\end{equation}
Then the symplectomorphism 
\begin{equation}\label{ML11}
\gamma_\phi: (x, \xi) \mapsto (x, \xi+d\phi(x))
\end{equation}
preserves $\Sigma_0$ and maps the zero section in $T^*\bbR^n$ onto the Lagrangian submanifold 
\begin{equation}\label{ML12}
\Lambda_\phi = \{(x, d\phi(x))\ : x \in \bbR^n \}.
\end{equation}
Moreover, $\gamma_\phi$ has a natural quantization, the semi-classical Fourier integral operator 
\begin{equation}
\label{ML13}
T_\phi f(x) = e^{\frac{i\phi}{\hbar}}f(x). 
\end{equation}
\noindent{\bf Claim:} This operator preserves $I^\bullet(\Sigma_0)$. 
\begin{proof}
Given $\Upsilon=a(t, \frac{u}{\sqrt \hbar}, \hbar) \in I(\Sigma_0)$ let $\Upsilon_1=e^{\frac{i\phi}{\hbar}}\Upsilon$. By (\ref{ML10}), 
\[
\phi(t, u) = \sum \frac{u_r}{\hbar^{1/2}} \frac{u_s}{\hbar^{1/2}} \psi_{r,s}(t, \hbar^{1/2}\frac{u}{\hbar^{1/2}}) = \psi(t, \frac{u}{\hbar^{1/2}}, \hbar),
\]
hence 
\begin{equation}\label{ML14}
\Upsilon_1=\Upsilon_2(t, \frac{u}{\hbar^{1/2}}, \hbar)
\end{equation}
where $\Upsilon_2(t, u, \hbar)=e^{i\sum u_r u_s \psi_{r,s}(t, \hbar^{1/2}u)}a(t,u, \hbar)$ and hence 
is in $I(\Sigma_0)$. 
\end{proof}
We also note for future reference that 
\begin{equation}\label{ML15}
\Upsilon_1(t, u, 0)=e^{i\sum u_r u_s \psi_{r,s}(t, 0)}a(t,u,0).
\end{equation}

\subsubsection{Invariance under partial Fourier transforms}

Let $T^*\bbR^n=T^*\bbR^k \times T^*\bbR^l$ and let $\gamma: T^*\bbR^n \to T^*\bbR^n$ be the symplectomorphism which is equal to the identity on $T^*\bbR^k$ and the map, $(u, \eta) \mapsto (-\eta, u)$ on $T^*\bbR^l$. This symplectomorphism maps the zero section, $\Sigma_0$ in $T^*\bbR^k$ identically onto itself and maps the zero section of $T^*\bbR^n$ onto the conormal bundle of $\Sigma_0$ in $T^*\bbR^n$. Moreover its quantization is the semiclassical Fourier transform 
\begin{equation}\label{ML16}
f(t, u) \mapsto \int e^{\frac{iu\mu}{\hbar}} f(t, \tilde u) d\tilde u
\end{equation}
in the variable $u$ with $t$ held fixed. In particular it maps $\Upsilon(t, u, \hbar)=a(t, \frac{u}{\sqrt{\hbar}}, \hbar)$ in $I^0(\Sigma_0)$ onto 
\begin{equation}\label{ML17}
\Upsilon_1(t, u, \hbar) =\hbar^{\frac l2} \hat a(t, \frac{u}{\sqrt \hbar}, \hbar)
\end{equation}
in $I^{\frac l2}(\Sigma_0)$ where $\hat a(t, u, \hbar)$ is the classical Fourier transform of $a$ in the variable $u$ with $\hbar$ and $t$ held fixed. 

We note for future reference that if $\sigma_t(\Upsilon)(u)=a(t, u, 0)$ and $\sigma_t(\Upsilon_1)(u)=\hat a(t, u, 0)$, then 
\begin{equation}\label{ML18}
\sigma_t(\Upsilon_1) = \int e^{iu \cdot \mu} \sigma_t(\mu) d\mu. 
\end{equation}

\subsubsection{Decomposition of symplectomorphisms preserving $\Sigma_0$}

We first describe the linear symplectomorphisms which preserves $\Sigma_0$. If $A: T^*\bbR^n \to T^*\bbR^n$ is a linear symplectomorphism and is the identity on $\Sigma_0$, then it is the identity on $T^*\bbR^n/\Sigma_0^o$ and hence is basically a linear symplectomorphism of $\Sigma_0^o/\Sigma_0=T^*\bbR^l$. As above let $(u_1, \cdots, u_l, \mu_1, \cdots, \mu_l)$ be cotangent coordinates on $T^*\bbR^l$. Then by a standard theorem in symplectic linear algebra the group of linear symplectomorphisms of $T^*\bbR^l$ is generated by linear mappings of type I, I\!I and I\!I\!I, i.e. linear maps of the form 
\begin{equation}\label{MLI}
\begin{pmatrix} B & 0 \\ 0 & (B^t)^{-1} \end{pmatrix},
\end{equation}
of the form 
\begin{equation}\label{MLII}
\begin{pmatrix} I & C \\ 0 & I \end{pmatrix},
\end{equation}
and of the form 
\begin{equation}\label{MLIII}
\begin{pmatrix} 0 & I \\ -I & 0 \end{pmatrix},
\end{equation}

We will prove below 
\begin{theorem}\label{MLT1}
If $\gamma: T^*\bbR^n \to T^*\bbR^n$ is a symplectomorphism which maps $\Sigma_0$ onto $\Sigma_0$ and is the identity on $\Sigma_0$, then it is a composition 
\begin{equation}\label{ML20}
\gamma = A \gamma_\phi \gamma_f \gamma_0,
\end{equation}
where $\gamma_0$ is a symplectomorphism whose restriction to the zero section of $T^*\bbR^n$ is the identity, $\gamma_f$ is the symplectomorphism (\ref{ML9}), $\gamma_\phi$ is the symplectomorphism (\ref{ML11}) and $A$ a linear symplectomorphism whose restriction to $\Sigma_0$ is the identity. 
\end{theorem}
\begin{proof}
Let $M=T^*\bbR^n$ and let $(TM)_{vert}$ be the vertical subbundle of $TM$, i.e. for each $p \in M$ the vectors in $(T_pM)_{vert}$ are vectors which are tangent at $p$ to the fiber at $p$ of the projection, $T^*\bbR^n \to \bbR^n$. Given $A \in Sp(\bbR^{2l})$, we get from $A^{-1}\gamma$, by restriction to the zero section, $\bbR^n$, of $M$ a map, $f_A: \bbR^n \to M$, a map, $df_A: T\bbR^n \to TM$ and as we vary $A$, a map 
\begin{equation}\label{ML21}
df: Sp(\bbR^{2l}) \times T\bbR^n \to TM.
\end{equation}
It is easy to see that this map is transversal to $(TM)_{vert}$, and hence, by Thom transversality, there exists an $A \in Sp(\bbR^{2l})$ such that $df_A$ is transversal to $(TM)_{vert}$, i.e. the graph of $f_A: \bbR^n \to T^*\bbR^n$ is horizontal. In particular, its image is a horizontal Lagrangian submanifold of $M$ of the form $\mathrm{Image}(\gamma_\phi|_{\bbR^n})$, where $\gamma_\phi$ is the symplectomorphism (\ref{ML11}). Thus 
\begin{equation}\label{ML22}
\gamma=A \gamma_\phi   \gamma_f   \gamma_0,
\end{equation}
where $\gamma_f$ is the symplectomorphism (\ref{ML9}), $\gamma_\phi$ the symplectomorphism (\ref{ML11}) and $\gamma_0$ a symplectomorphism which is the identity on the zero section of $T^*\bbR^n$. 
\end{proof}

Finally we note that if $A$ is a linear symplectomorphism of the form (\ref{MLI}) it is a symplectomorphism of type $\gamma_f$, if it is of the form (\ref{MLII}) it is of type $\gamma_\phi$, and if it is of the form (\ref{MLIII}) its quantization is the partial Fourier transform (\ref{ML16}).

\subsubsection{Proof of theorem \ref{MLT2}}
By theorem \ref{MLT1} it suffices to prove this for the quantization (\ref{ML13}) of $\gamma_\phi$, the quantization $f^*$ of $\gamma_f$, the quantization (\ref{ML16}) of the linear symplectomorphism (\ref{MLIII}) and all quantizations of $\gamma_0$. However, for the quantizations of $\gamma_\phi$, $\gamma_f$ and the symplectomorphism (\ref{MLIII}) that we've just alluded to, this assertion follows from lemma \ref{gfHtV}.
\hfill{$\Box$}

\subsection{Transformation of the symbols under FIOs preserving $\Sigma_0$}

Given an element, $\Upsilon=\hbar^l a(t, \frac{u}{\sqrt \hbar}, \hbar)$ of $I^l(\Sigma_0)$ we have defined its symbol 
\[
\sigma(\Upsilon)_t \in \mathcal S(\bbR^l)
\]
at $t \in \Sigma_0$ to be the Schwartz function 
\[
\sigma(\Upsilon)(t) = a_t(u) := a(t, u, 0).
\]
Next we will discuss the ``symbolic calculus" of these symbols, i.e. describe how they transform under composition by the FIO's in theorem \ref{MLT2}. In view of the factorization, $\gamma=A \gamma_\phi \gamma_f \gamma_0$ in theorem \ref{MLT1} it suffices to describe how they transform for the FIO's (\ref{ML6}), (\ref{ML13}) and (\ref{ML16}) and for the pullback map 
\[
f^*: \mathcal S(\bbR^n) \to \mathcal S(\bbR^n)
\]
quantizing $\gamma_f$. However, for the FIO (\ref{ML6}) we showed that symbols transform by the formula (\ref{ML8}), for the FIO (\ref{ML13}) by the formula (\ref{ML15}) and for the FIO (\ref{ML16}) by the formula (\ref{ML18}). Moreover, in addition it's easy to check that for the pull-back map $f^*$, symbols transform by the recipe 
\begin{equation}\label{ML23}
\sigma_t(f^*\Upsilon) (u)= \sigma_{f(t)}(\Upsilon)((df_t)^{-1}(u)).
\end{equation}
These formulas, by the way, have a nice abstract interpretation which we'll discuss later in this paper. Namely if $\gamma: T^*\bbR^n \to T^*\bbR^n$ is a symplectomorphism whose restriction to $\Sigma_0$ is the identity then for $t \in \Sigma_0$ the linear map $(d\gamma)_t$ restrict to a linear symplectomorphism of the symplectic normal bundle to $\Sigma_0$ at $t$, or, in other words, an element $(L_\gamma)_t$ of $Sp(\bbR^{2l})$; and for the Fourier integral operator, $F_\gamma$, quantizing $\gamma$ in each of the cases above 
\begin{equation}\label{ML24}
\sigma_t(F\Upsilon) =(L_\gamma)_t^\# \sigma_t(\Upsilon),
\end{equation}
where $(L_\gamma)_t^\#$ is the metaplectic representation of $(L_\gamma)_t$ on the Schwartz space $\mathcal S(\bbR^l)$.

\section{Global theory}
\newcommand{\Mp}{{\rm Mp}}
\newcommand{\Sp}{{\rm Sp}}
\newcommand{\Gl}{{\rm Gl}}

We begin with the global definition on manifolds.  We will keep the notation 
of \S 2 for the model spaces $I^r(\Sigma_0)$.

\begin{definition}\label{generaldef}
Let $M$ be a smooth manifold of dimension $n$, and $\Sigma\subset T^*M$ an isotropic 
submanifold of
its cotangent bundle, of codimension $n+l$.  
Let $r$ be a half-integer.  Then the space $I^r(\Sigma)$ is defined
as the set of all $\h$-dependent functions, $\Upsilon:M\times (0,\h_0)\to\bbC$
with wave-front set contained in $\Sigma$, such that there exists a microlocal partition of unity
$\{\chi_\ell\}$ of a neighborhood of $\Sigma$, and zeroth order semiclassical
Fourier integral operators $F_\ell$, from some open sets $\calU_\ell\subset\bbR^n$ to $M$, such that
\begin{equation}\label{}
\chi_\ell(\Upsilon) = F_\ell (\Upsilon_\ell) + O(\h^\infty),
\end{equation}
where $\Upsilon_\ell\in I^r(\Sigma_0)$ is supported in $\calU_\ell$ and $F_\ell$ is associated
to a canonical transformation mapping $\Sigma_0\cap T^*\calU_\ell$ diffeomorphically 
onto a relative open set in $\Sigma$.
\end{definition}

\bigskip
The rest of this section is devoted to the global definition of the symbol of
an element in $I^r(\Sigma)$, and to the symbol calculus under the action of pseudodifferential
operators, including the transport equations.

\subsection{The metaplectic representation \`a la Blattner-Kostant-Sternberg}

A special case of the work of Blattner, Kostant and Sternberg on geometric quantization and representation 
theory is a construction of the metaplectic representation that we now review.  For the 
original exposition, see \cite{Bla} and \cite{Bla1}.  This material will play a 
crucial role in the definition
of the intrinsic symbol of an isotropic state.

\smallskip
We begin by recalling that one can quantize a symplectic vector space $(V,\omega)$ 
by choosing a metaplectic structure on it and a lagrangian subspace $L\subset V$.
The result is a Hilbert space, $\calH_L$.  Its elements are sections of the (trivial)
pre-quantum bundle of $V$ tensored with half-forms transverse to the translates
of $L$, which are covariantly constant and square-integrable over the
quotient $V/L$.  
The point we want to underline here is that the construction is covariant
with respect to metaplectic linear maps.  More precisely, if $V'$ is another metaplectic
vector space, $L'\subset V'$ a Lagrangian submanifold, and $g:V\to V'$ a metaplectic
isomorphism mapping $L$ to $L'$, then the inverse of the
natural pull-back operator induces a unitary operator
\begin{equation}\label{ug}
U^{g}_L: \calH_{L}\to \calH'_{L'}.
\end{equation}
In particular, if $g\in\Mp(V,\omega)$ (the metaplectic automorphisms of $(V,\omega)$), 
then the action of $g$ on $V$ induces a unitary operator
\begin{equation}\label{}
U^g_L: \calH_L\to \calH_{g(L)}.
\end{equation}
It is evident that one has the cocycle condition
\begin{equation}\label{}
U^{g'}_{g(L)}\circ U^g_L = U^{g'g}_L.
\end{equation}

\medskip
With this natural construction at hand,
the key ingredient needed in the construction of the metaplectic representation
is the BKS pairing:  Given two lagrangian subspaces $L,L'\subset V$, there is a 
sesquilinear pairing
\begin{equation}\label{}
(\cdot , \cdot) : \calH_L\times \calH_{L'}\to\bbC
\end{equation}
which in fact corresponds to a unitary operator 
\begin{equation}\label{}
V_{L',L}:  \calH_L\to \calH_{L'}
\end{equation}
in the sense that
\begin{equation}\label{}
\forall \psi\in\calH_L,\ \psi'\in\calH_{L'}\quad (\psi, \psi') = \inner{V_{L',L}(\psi)}{\psi'}_{\calH_{L'}}.
\end{equation}
These unitary operators also satisfy a cocycle condition:
\begin{equation}\label{}
V_{L'',L'}\circ V_{L',L} = V_{L'',L}.
\end{equation}
In addition, the pairing and the action of Mp satisfy a naturality condition:  Given $L,L'\subset V$ Lagrangians, and $g\in\Mp$, it's clear 
from the definitions that
\begin{equation}\label{}
\forall \psi\in\calH_L,\ \psi'\in\calH_{L'}\quad (U^g_L(\psi), U^g_{L'}(\psi')) =(\psi, \psi'),
\end{equation}
and from this it follows that the following diagram commutes:
\begin{equation}\label{}
\begin{array}{rcccl}
 & \calH_L & \xrightarrow{V_{L',L}} & \calH_{L'} & \\
 U^g_L & \downarrow & & \downarrow & U^g_{L'}\\
  & \calH_{g(L)} &\xrightarrow{V_{g(L'), g(L)}} & \calH_{g(L')} & \\
\end{array}
\end{equation}

\medskip
These two constructions together give the metaplectic representation:

\begin{definition}
For every $g\in\Mp(V,\omega)$ and a Lagrangian subspace $L\subset V$, define
\begin{equation}\label{}
\Mp_L(g): \calH_L \to \calH_L, \quad \Mp_L(g) = V_{L, g(L)} \circ U^g_L.
\end{equation}
\end{definition}
\begin{lemma}
With the previous notations, one has: $\Mp_L(g')\circ\Mp_L(g)=\Mp_L(g'g)$.
\end{lemma}
\begin{proof} (Sketch.)
Consider the commutative diagram:
\begin{equation}\label{uvscommute}
\begin{array}{rcccl}
 & \calH_{g(L)} & \xrightarrow{V_{L,g(L)}} & \calH_{L} & \\
 U^{g'}_{g(L)} & \downarrow & & \downarrow & U^{g'}_{L}\\
  & \calH_{g'g(L)} &\xrightarrow{V_{g'(L), g'g(L)}} & \calH_{g'(L)} & \\
\end{array}
\end{equation}
and use it to flip the middle  $U$ and $V$ in the composition $\Mp_L(g')\circ\Mp_L(g)$.
Then use the cocycle conditions.
\end{proof}
It is known that $\Mp_L$ is the metaplectic representation.  

\medskip
Note that if $L,\Lambda\subset V$ are transverse Lagrangians, then there is an isomorphism
\begin{equation}\label{identL}
\calH_L \cong L^2(\Lambda),
\end{equation}
where the right-hand side is the Hilbert space of half forms on $\Lambda$ with respect to the 
metalinear structure on $\Lambda$ inherited from the metaplectic structure on $V$.
Thus we obtain the metaplectic representation on $L^2(\Lambda)$, arising from (\ref{identL})
and the representation $\Mp_L$ of $\Mp(V)$ on $\calH_L$.
However, if $L'$ is another Lagrangian subspace transverse to $\Lambda$, then $\Mp_{L'}$ 
induces a {\em different} (though of course isomorphic) metaplectic representation of the same group,
$Mp(V)$, on $L^2(\Lambda)$.

\bigskip
We can use the previous results to define {\em an abstract Hilbert space associated to the metaplectic
vector space} $V$, as follows:
\begin{definition}
If $V$ is a metaplectic vector space, define
\begin{equation}\label{}
\calH^V = \{(L, \psi)\;;\; L\subset V\ \text{Lagrangian subspace and } \psi\in\calH_L\}/\sim
\end{equation}
where
\begin{equation}\label{}
(L,\psi)\sim (L',\psi')\quad\Leftrightarrow\quad \psi' = V_{L', L}(\psi),
\end{equation}
with the norm
\begin{equation}\label{}
\norm{[(L,\psi)]} = \norm{\psi}_{\calH_L}.
\end{equation}
We will denote by 
\begin{equation}\label{}
\calS^V\subset \calH^V
\end{equation}
the image of the space of smooth vectors of any $\calH_L$ (under the Heisenberg
representation). 
\end{definition}

\medskip

One can easily check the following:
\begin{lemma}
Let $V$ be a metaplectic vector space, $\Psi=[(L,\psi)]\in\calH_V$ and $g\in\Mp(V)$.  Then the element
\begin{equation}\label{}
\Mp(g)(\Psi):=[(L, \Mp_L(\psi))]\in \calH_V
\end{equation}
is well-defined, and $g\mapsto \Mp(g)$ is the metaplectic representation of $\Mp(V)$
on the abstract space $\calH_V$.
\end{lemma}
By a slight generalization of (\ref{uvscommute}), 
we get that a metaplectic map $f:V\to V'$ induces a unitary operator
\begin{equation}\label{metainduced}
U^f: \calH^V\to\calH^{V'}
\end{equation}
by choosing any $L\subset V$ and considering $U^f_L:\calH^V_L\to \calH^{V'}_{L'}$ with
$L' = f(L)$.

\bigskip
{
In what follows we will also need the representation of the Heisenberg
group of $V$ on the abstract Hilbert space $\calH^V$.  
It is known that, for each polarization $L\subset V$, there is a unique (up to isomorphism)
representation of the Heisenberg group on $\calH^V$ such that the Lie algebra
element $(0,1)\in V\oplus\bbR$ acts as multiplication by $\sqrt{-1}$.  It is also known that
the metaplectic representation intertwines these representations (for different choices
of $L$), therefore the Heisenberg representation is well-defined on the abstract
Hilbert space $\calH^V$.  
}

\subsection{Symbols of isotropic states}

Let $M$ be a smooth manifold of dimension $n$, 
and $\Sigma\subset T^*M$ an isotropic submanifold of codimension $n+l$.  
We will denote by $\calN^\Sigma\to \Sigma$ the vector bundle (of rank $2l$)
whose fiber at $s\in \Sigma$ is the symplectic normal vector space
\begin{equation}\label{}
\calN_s^\Sigma := \left( T_s\Sigma\right)^\circ/ T_s\Sigma.
\end{equation}
In order to have a global notion of the symbol of an element $\Upsilon\in I^r(\Sigma)$,
we need to assume that $\calN^\Sigma\to \Sigma$ has a metaplectic structure and
we need to choose one such structure.  We will proceed henceforth under this assumption.

\medskip
We pick once and for all a metaplectic structure on $\bbR^{2n}$, 
which induces
a metaplectic structure on the symplectic normals $\calN_s^{\Sigma_0}$.

{{
If $V$ is a metaplectic vector space, $\wedge^{1/2}V$ will
denote the one-dimensional space 
of half-forms on $V$, that is, 
functions $\psi$ on the space of metaplectic
frames $m$ of $V$ that transform according to the rule
$\psi(g\cdot m) = \det^{1/2}(g)\psi(m)$, for all metaplectic linear maps $g$.} }

\medskip
We first define the symbol of a model state:

\begin{definition}
Let $\Upsilon\in I^r(\Sigma_0)$ be given by equation (\ref{basicdef}).
Then:
\begin{enumerate}
\item The \underline{model symbol}
of $\Upsilon$ at $s=(t, 0;0,0)\in\Sigma_0$ is
\begin{equation}\label{babysymbol}
\tilde\sigma_\Upsilon (s) = a_0(t,\cdot)\,
{(dt_1\wedge\cdots\wedge dt_k)^{1/2} 
(du_1\wedge\cdots\wedge du_l)^{1/2}}\in 
\wedge^{1/2}\bbR^k\otimes\calS(\bbR^l),
\end{equation}
{where $\calS(\bbR^l)$ now denotes the 
space of Schwartz half-forms on $\bbR^l$.}
\item Note that, in a canonical way, for any $s\in\Sigma_0$,
\begin{equation}\label{}
\calN_s^{\Sigma_0}\cong \bbR^{2l}.
\end{equation}
In particular, there is a canonical polarization (lagrangian subspace) 
in all normal spaces, $L_0\subset\calN^{\Sigma_0}_s$,
arising from the vertical polarization of $T^*\bbR^{2n}$, which gives us
an isomorphism
\begin{equation}\label{}
\calH^{\calN^{\Sigma_0}_s}_{L_0} = L^2(\bbR^l),
\end{equation}
{with $L^2(\bbR^l)$ denoting now
the space of square integrable
half forms on $\bbR^l$.}
\end{enumerate}
We define the \underline{symbol} of $\Upsilon$ at $s$ to be the element
\begin{equation}\label{}
\sigma_\Upsilon (s) \in \calS(\calN^{\Sigma_0}_s)\otimes 
{\wedge^{1/2}}T_{(t,0)}\Sigma_0
\end{equation}
represented by $\tilde\sigma_\Upsilon(s)$, in the abstract space of smooth vectors
of the quantization of the symplectic normal.
\end{definition}

\medskip
In the manifold case the symbols will be ``transplanted" from the model case by 
Fourier integral operators.
That this is possible follows from the following Lemma, which in fact we have already
proved in \S 2:

\begin{lemma}\label{symbcalcbasis}
Let $\Upsilon\in I^r(\Sigma_0)$ and $F$ a Fourier integral operator from $\bbR^n$ to itself
associated to a transformation $f:T^*\bbR^n\to T^*\bbR^n$
that preserves $\Sigma_0$ (set-wise).  Let $s\in\Sigma$, and let
\begin{equation}\label{}
\varphi: \calN_s^{\Sigma_0}\to\calN_{f(s)}^{\Sigma_0}
\end{equation}
be the symplectomorphism induced by the differential $df_{s}$.  Assume that it lifts
to an Mp map, so that we have a metaplectic operator 
\begin{equation}\label{}
\Mp(\varphi): L^2(\bbR^l)\to L^2(\bbR^l).
\end{equation}
Then
\begin{equation}\label{}
\tilde\sigma_{F(\Upsilon)}(f(s)) = \Mp(\varphi)(\tilde\sigma_\Upsilon)(s)\otimes \nu
\end{equation}
where $\nu$ is {the image of the half-form
factor of the symbol of $\Upsilon$ times the symbol of $F$,
under the canonical map
\begin{equation}\label{halfcanmap}
\forall \sigma\in\Sigma_0\qquad
\wedge^{1/2}T_s\Sigma_0\otimes \wedge^{1/2}T\Gamma_{(f(s),s)} \to
\wedge^{1/2}T_{f(s)}\Sigma_0,
\end{equation}
where $\Gamma$ is the graph of $f$.}
\end{lemma}
{
For completeness we mention that (\ref{halfcanmap}) arises,
as in the composition of Fourier integral operators and
Lagrangian distributions,
from the short exact sequence
\[
0\to T_{f(s)}\Sigma_0\to T\Gamma_{(f(s),s)} \oplus T_s\Sigma_0
\to \bbR^{2n}
\]
where the first map is $v\mapsto ((v, df^{-1}(v)), df^{-1}(v))$
and the second is $((v,w),w_1)\mapsto w-w_1$, see \cite{GSGeometricA} and \S 6 of \cite{BG} (though the present case
is much simpler because $f$ is a {\em transformation}).
}

\bigskip
{
We now pass to the manifold case.  We refer to
\S 5 of \cite{BG} for general background 
of metaplectic structures on cotangent
bundles, and how they arise from metalinear structures on the base.}

\medskip
Let $\calU\subset\bbR^n$ open and
$f:T^*\calU\to T^*M$ be a symplectic embedding mapping $T^*\calU\cap\Sigma_0$
onto a relatively open set of $\Sigma$.  Let us further assume that $f$ is an Mp map, 
in the sense that the maps induced by the differential $df$:
\begin{equation}\label{mpmaps}
\varphi_{s_0}: \calN_{s_0}^{\Sigma_0} \to \calN_s^\Sigma,\quad s=f(s_0)
\end{equation}
are metaplectic maps, $\forall s_0\in T^*\calU\cap\Sigma_0$.

\begin{corollary}
Let $F_j: C^\infty(\calU_j)\to C^\infty (M)$, with $j=1,2$, be FIOs associated with 
canonical embeddings $f_j: T^*\calU_j\to T^*M$, as above.  
Let $\Upsilon_j\in I^r(\Sigma_0)$
have support in $\calU_j$, and assume that
\begin{equation}\label{}
F_1(\Upsilon_1) = F_2(\Upsilon_2)\quad \text{mod}\ I^{r+1/2}(\Sigma).
\end{equation}
Let $ s_j\in T^*\calU_j\cap\Sigma_0$ be such that $f_1(s_1) = s = f_2(s_2)$, and let
$\varphi_j: \calN^{\Sigma_0}_{s_j}\to \calN_s^\Sigma$ be the corresponding
(metaplectic) maps (\ref{mpmaps}).

\underline{Then}, with the notation (\ref{metainduced}),
\begin{equation}\label{symbolequality}
U^{\varphi_1}(\sigma_{\Upsilon_1}(s_1)) = U^{\varphi_2}(\sigma_{\Upsilon_2}(s_2)).
\end{equation}
\end{corollary}
\begin{proof}
Let us define $g$ to have a commutative diagram of Mp maps:
\begin{equation}\label{}
\begin{array}{ccc}
 &\calN_s^\Sigma& \\
\varphi_1\nearrow & & \nwarrow \varphi_2\\
\calN_{s_1}^{\Sigma_0}= \bbR^{2k} &\xrightarrow{g} & \bbR^{2k}= \calN_{s_2}^{\Sigma_0}
\end{array}
\end{equation}
$\varphi_j$ maps the vertical polarization $L_0$ to a polarization $L_j$.  We need to show
that the images of $\tilde\sigma_j$ under
\begin{equation}\label{}
U^{\varphi_j}_{L_0}:  L^2(\bbR^k) \to \calH_{L_j}
\end{equation}
with $j=1,2$ represent the same element of the abstract Hilbert space of $\calN_s^\Sigma$.
(Recall that $L_0\subset \Sigma_0$ denotes the vertical polarization.)

Let us denote by $\calH_{L_j}$ the quantization of $\calN_s^\Sigma$ with respect to the polarization
$L_j$.  Since $\varphi_1 = \varphi_2\circ g$,
\begin{equation}\label{}
U^{\varphi_1}_{L_0} =U^{\varphi_2}_{g(L_0)}\circ U^g_{L_0},
\end{equation}
which we rewrite as
\begin{equation}\label{}
U^{\varphi_1}_{L_0} =\left[U^{\varphi_2}_{g(L_0)}\circ V_{g(L_0),L_0}\right] \circ
\left[ V_{L_0,g(L_0)}\circ U^g_{L_0} \right] = \left[U^{\varphi_2}_{g(L_0)}\circ V_{g(L_0),L_0}\right]
\circ \Mp(g)_{L_0}.
\end{equation}
Apply both sides to the Schwartz function $\tilde\sigma_{\Upsilon_1}(s)$.
We make two replacements on the right-hand side of the resulting equality.
First, by Lemma \ref{symbcalcbasis}, $\Mp(g)_{L_0}(\tilde\sigma_{\Upsilon_1}(s)) =
\tilde\sigma_{\Upsilon_2}(s)$.  Second, by naturality,
\begin{equation}\label{}
U^{\varphi_2}_{g(L_0)}\circ V_{g(L_0),L_0} = V_{L_1,L_2}\circ U^{\varphi_2}_{L_0}
\end{equation}
We conclude that
\begin{equation}\label{}
U^{\varphi_1}_{L_0}(\tilde\sigma_{\Upsilon_1}(s)) =
V_{L_1,L_2}\left(U^{\varphi_2}_{L_0}(\tilde\sigma_{\Upsilon_2}(s))\right).
\end{equation}
But this shows that $U^{\varphi_1}_{L_0}(\tilde\sigma_{\Upsilon_1}(s))$
and $U^{\varphi_2}_{L_0}(\tilde\sigma_{\Upsilon_2}(s))$ represent the same
element in the abstract quantization of $\calN_s^\Sigma$.
\end{proof}

This Corollary allows us to make the following
\begin{definition}\label{GlobalSymbolDef}
Let $\Upsilon\in I^r(\Sigma)$ be given by
$
\Upsilon = F(\Upsilon_0),
$
where $\Upsilon_0\in I^r(\Sigma_0)$ and $F$ is a zeroth-order FIO associated to a
canonical transformation $f$, as in Definition
\ref{generaldef}.  Then the symbol of $\Upsilon$ at $s\in\Sigma$ 
{is the element 
\begin{equation}\label{}
\sigma_{\Upsilon}(s)\in \calH^{\calN_s^\Sigma}\otimes \wedge^{1/2}
T_s\Sigma^{1/2}
\end{equation}
which is the
image of the symbol of $\Sigma_0$ at $s_0:=f^{-1}(s)$ under the map $\varphi$ 
induced by $df_s$, tensored with the image of the symbol of $F$
and the half-form part of the symbol of $\Upsilon_0$,
under the generalization of (\ref{halfcanmap})
\begin{equation}\label{halfcanmapgen}
\forall \sigma\in\Sigma_0\qquad
\wedge^{1/2}T_s\Sigma_0\otimes \wedge^{1/2}T\Gamma_{(f(s),s)} \to
\wedge^{1/2}T_{f(s)}\Sigma.
\end{equation}
}
We extend this definition to a general $\Upsilon\in I^r(\Sigma)$ by linearity.
\end{definition}

\medskip
{
Note that the symbol of $\Upsilon\in I^r(\Sigma)$ can be regarded as a 
section of an infinite-rank bundle
over $\Sigma$, with fibers $\calH^{\calN_s^\Sigma}\otimes 
\wedge^{1/2} T_s\Sigma^{1/2}$.
These are the {\em symplectic spinors} of \cite{Gui}.
}

\subsection{The symbol calculus}

In this section we re-interpret the results of \S 2.2 in the language of the
global symbol.

\begin{theorem}\label{symbolcalc0}
Let $A$ be a semiclassical $\Psi$DO of order $m$ on $M$, and 
$\Upsilon\in I^r(\Sigma)$.  Then $A(\Upsilon)\in I^{r+m}(\Sigma)$,
and its symbol is simply the pointwise product $\left(\alpha |_\Sigma\right)\sigma_\Upsilon$,
where $\alpha$ is the principal symbol of $A$.
\end{theorem}
 
In the remainder of this section we assume that
\begin{equation}\label{nullAssumption}
\alpha |_\Sigma \equiv 0.
\end{equation}
Let $\Xi$ denote the Hamilton vector field of $\alpha$.   Since $\Sigma$ is isotropic,
\begin{equation}\label{}
\forall s\in\Sigma\qquad \Xi_s \in \left(T_s\Sigma\right)^\circ,
\end{equation}
and therefore $\Xi_s$ projects to a vector $\xi_s\in \calN_s^\Sigma$.
It therefore defines an element $(\xi_s,0)$ in the Lie algebra $\calN_s^\Sigma\oplus\bbR$
of the Heisenberg group of $\calN_s^\Sigma$, which we continue to denote by $\xi_s$.

\subsubsection{Transport equations}
\begin{theorem} \label{symbolcalc1}(1st transport equation)
In the situation of Theorem \ref{symbolcalc0}, assume (\ref{nullAssumption}).
Then $A(\Upsilon) \in I^{r+m+1/2}(\Sigma)$, and its symbol is 
\begin{equation}\label{1sttransport}
\sigma_{A(\Upsilon)}(s) = d\rho(\xi_s)(\sigma_\Upsilon(s))
\end{equation}
where $d\rho$ is the infinitesimal Heisenberg representation of the Heisenberg 
group of $\calN_s^\Sigma$.
\end{theorem}

Next we consider the case when, in addition to (\ref{nullAssumption}), one has:
\begin{equation}\label{nullAssumptionStronger}
\forall s\in\Sigma\qquad \Xi_s \in T_s\Sigma,\quad\text{that is to say}\quad \xi_s=0.
\end{equation}
In that case the right-hand side of (\ref{1sttransport}) is zero, and
$A(\Upsilon)\in I^{r+m+1}(\Sigma)$.  

\medskip
To compute its symbol, let us introduce the flow $\phi_r: \Sigma\to\Sigma$ of
the restriction of $\xi$ to $\Sigma$.  Since $\phi_r$ is the restriction of a Hamiltonian
flow on $T^*M$, it has a natural lift $\Phi_r$ to the symplectic normal bundle
\begin{equation}\label{}
\begin{array}{ccc}
\calN^\Sigma & \xrightarrow{\Phi_r} & \calN^\Sigma \\
\downarrow & & \downarrow\\
\Sigma & \xrightarrow{\phi_r} & \Sigma
\end{array}
\end{equation}
which is a symplectomorphism fiber-wise.  Since $\Phi_0$ is the identity,
$\Phi_r$ has a natural lift to the Mp structure of the symplectic normal.  Therefore
we get unitary operators:
\begin{equation}\label{}
U^r_s: \calH^{\calN^\Sigma_s} \longrightarrow \calH^{\calN^\Sigma_{\phi_r(s)}} .
\end{equation}

\begin{theorem}\label{symbolcalc2}(2nd transport equation)
In the situation of Theorem \ref{symbolcalc0}, assume (\ref{nullAssumption})
and (\ref{nullAssumptionStronger}).
Then $A(\Upsilon) \in I^{r+m+1}(\Sigma)$, and its symbol at $s\in\Sigma$ is 
\begin{equation}\label{secondorder}
\sigma_{A(\Upsilon)}(s) = \frac{1}{\sqrt{-1}} 
\frac{d\ }{dr} U_s^{-r}(\sigma_\Upsilon(\phi_r(s)))|_{r=0} 
{+\sigma^{\text{sub}}_A\,\sigma_\Upsilon,}
\end{equation}
{where $\sigma^{\text{sub}}_A$ denotes the subprincipal symbol
of $A$.}
\end{theorem}
We can think of (\ref{secondorder}) as a Lie derivative on the infinite-rank bundle
over $\Sigma$ with fibers $\calH^{\calN^\Sigma_s}$.  The Lie derivative exists because of the  
existence of the natural lifts of $\phi_r$ to the bundle automorphisms $U^r$, as explained above.  
{
\begin{proof}
By the manifest covariance of (\ref{secondorder}) with respect to 
invertible FIOs, it suffices to prove it in the model case.
Consider again $\bbR^n=\bbR^k\oplus\bbR^l$, with coordinates
$x =(t,u)$, and the model isotropic
$\Sigma_0 = \{(t,0;0)\}\subset T^*\bbR^n$,  
and fix $s=(t_0,0;0)\in\Sigma_0$.
Consider also a simple model isotropic state, 
$\Upsilon(t,u,\h) = a(t,u\h^{-1/2})$. Let us identify all Hilbert
spaces $\calH^{\calN^\Sigma_s}$ with $L^2(\bbR^l)$.
Then, by the discussion of \S 3.1, we have:
\begin{equation}\label{steak}
U_s^{-r}(\sigma_\Upsilon(\phi_r(s))) = 
Mp(g_r)(a(T(r,s),\cdot))\;\phi_r^*
((\wedge^{1/2} dt)
(\wedge^{1/2} du))
\end{equation}
where $g_r$ is the linear symplectic transformation induced
by $d\phi_r$ in the symplectic normal, and 
$T(r,s)$ is the value of the $t$ coordinate at
$\phi_r(s)$.  By Leibniz' rule, upon differentiation of (\ref{steak})
we obtain 
the sum of three terms: 
\begin{equation}\label{firsterm}
I = mp(\chi)(a(t_0,\cdot))\,(\wedge^{1/2} dt)(\wedge^{1/2} du),
\end{equation}
where $\chi = \frac{d\ }{dr}dg_r|_{r=0}$ and $mp$ is the
infinitesimal metaplectic representation;
\begin{equation}\label{secondterm}
I\!I = \frac{d\ }{dr} a(T(r,s),\cdot)|_{r=0} \,(\wedge^{1/2} dt)(\wedge^{1/2} du);
\end{equation}
and
\begin{equation}\label{thirdterm}
I\!I\!I = a(t_0,\cdot)\,\frac 12\, \left(\sum 
\frac{\partial^2 H\ }{\partial t_i\partial\tau_i} +
\frac{\partial^2 H\ }{\partial u_j\partial\mu_j}
\right)
(\wedge^{1/2} dt)(\wedge^{1/2} du),
\end{equation}
where $H$ is the principal symbol of $A$.
(The calculation of the Lie derivative 
$\frac{d\ }{dr}\,\phi_r^*
((dt_1\wedge\cdots\wedge dt_k)^{1/2} 
(du_1\wedge\cdots\wedge du_l)^{1/2})|_{t=0}$, yielding
(\ref{thirdterm}), is exactly as in
in the proof of Proposition 1.3.1 in \cite{GSSemiclassical}.)
Let us now omit the half-form factors in $I$, $I\!I$ and $I\!I\!I$, and
show that, 
$\frac{1}{\sqrt{-1}}(I+I\!I+I\!I\!I)+\sigma^{\text{sub}}_A\,a(t_0,\cdot)$ 
equals (\ref{localTransport2}) (with the notational change $p_0=H$). 
The term $\frac{1}{\sqrt{-1}}I$ gives exactly the second line of 
(\ref{localTransport2}).  Further, given the assumption that the Hamilton
field of $H$ is tangent to $\Sigma_0$, 
\[
\frac{1}{\sqrt{-1}}\,I\!I = \frac 1{\sqrt{-1}} \sum  
\frac{\partial H}{\partial \tau_j}(s) \frac{\partial a}{\partial t_j}(t,u),
\]
so all that remains to be verified is that
\[
\frac{1}{\sqrt{-1}}\,\frac 12\, \left(\sum 
\frac{\partial^2 H\ }{\partial t_i\partial\tau_i} +
\frac{\partial^2 H\ }{\partial u_j\partial\mu_j}
\right)
+\sigma^{\text{sub}}_A\, = 
p_1(s).
\]
But this is equivalent to the expression for $\sigma^{\text{sub}}_A$ in
coordinates (see e.g. \S 1.3.4 in \cite{GSSemiclassical}).
\end{proof}
}

\subsubsection{Isotropic regularity}

We conclude this section with the observation 
that our spaces of isotropic states satisfy a certain  
{\em isotropic regularity} condition.

Let $\Sigma\subset T^*M$ an isotropic submanifold.  We let
\begin{equation}\label{}
\calP(\Sigma)
= \{f\in C^\infty(T^*M)\;;\; f|_\Sigma\equiv 0 \ \text{and}\ \Xi_f\ \text{is tangent to}\ \Sigma \}.
\end{equation}

\begin{lemma}
$\calP(\Sigma)$ is a Poisson subalgebra of $C^\infty(T^*M)$ and an ideal as well.
\end{lemma}
\begin{proof}
Let $f,g\in\calP(\Sigma)$.  We need to show that $\PB{f}{g}\in\calP(\Sigma)$.  Since
$\PB{f}{g} = \calL_{\Xi_f}g$, $\Xi_f$ is tangent to $\Sigma$ and $g$ is constant
on $\Sigma$, $\PB{f}{g}|_\Sigma =0$.  Also $\Xi_{\PB{f}{g}} = [\Xi_f, \Xi_g]$,
so this is tangent to $\Sigma$ if both $\Xi_f$, $\Xi_g$ are.

To prove the second part let $f\in\calP(\Sigma)$ and $g\in C^\infty(T^*M)$.  Clearly $fg$ vanishes
on $\Sigma$, and since, at any point on $\Sigma$
$\ \Xi_{fg} = g\Xi_f$, $\Xi_{fg}$ is tangent to $\Sigma$.
\end{proof}

The following is immediate from the symbol calculus:
\begin{corollary}
Given an isotropic $\Sigma\subset T^*M$,
the spaces $I^r(\Sigma)$ are stable under the action of 
arbitrary compositions
\[
P_1\circ\cdots \circ P_N
\]
where $P_1,\ldots ,P_N$ are first-order semiclassical pseudodifferential
operators whose symbols are all in
$\calP(\Sigma)$.
\end{corollary}
Thus elements in $I^r(\Sigma)$ ``don't get worse" upon application
of any number of first-order operators, provided their symbols are in
$\calP(\Sigma)$.
We conjecture that indeed the spaces $I^r(\Sigma)$ can be characterized
by such an isotropic regularity condition, in a
similar way to H\"ormander's characterization of lagrangian distributions
in \cite{Ho} (but in the semiclassical setting).

\subsection{Norm estimates}
This short section is devoted to the proof of the following
\begin{theorem}\label{NormEst}
Let $\Sigma\subset T^*X$ an isotropic of dimension $n-l$, $n = \dim (X)$,
and $\Upsilon\in I^r(\Sigma)$ with compact support.
Then
\begin{equation}\label{normEst}
\norm{\Upsilon}^2 = \h^{2r+l/2} \int_\Sigma |\sigma_\Upsilon|^2 + O(\h^{2r+l/2+1/2})
\end{equation}
where $|\sigma_\Upsilon|^2$ is the top-degree form on $\Sigma$
obtained by integrating, at each point $s\in \Sigma$,
the norm squared of the Schwartz function $\sigma_\Upsilon(s)$
(times the square of the half-form factor along $\Sigma$).
\begin{proof}
It suffices to verify the formula in the model case, that is, when
\[
\Upsilon =  a(t,\hmtwo u,\h): \bbR^n\times (0,\h_0)\to\bbC,
\]
where 
$
a(t,u,\h) \sim \h^{r}\sum_{j=0}^\infty a_j(t,u)\, \h^{j/2}
$
and we further assume that $\Upsilon$ is compactly supported
in the $t$ variables.  The proof is then elementary, reducing to the 
calculation of the leading term
\[
\h^{2r}\iint |a_0(t,\hmtwo u)|^2\, du\, dt = \h^{2r+l/2}\iint |a_0(t,v)|^2\, dv\, dt.
\]
\end{proof}
\end{theorem}

\subsection{An alternate approach}

We will conclude this description of isotropic states by showing that there 
is an alternate description of these states involving the Hermite 
distributions  of Boutet de Monvel, \cite{Bout}.
This alternate description follows \cite{PU2}.

\subsubsection{Local theory}

Let $U\subset\bbR^n$ be an open set.  Consider $T^*(U\times\bbR)$ with coordinates
$(x,\theta\;;\; \xi,\kappa)$, and let
\[
T^*(U\times\bbR)_+ = \{(x,\theta\;;\; \xi,\kappa)\;;\; \kappa>0\}.
\]
We begin by noticing that if $\frakh \in C_0^{-\infty}(U\times\bbR)$ is a distribution whose 
wave-front set is contained in $T^*(U\times\bbR)_+$, then 
the partial Fourier transform
\begin{equation}\label{deBorelization}
\widehat\frakh(x,\h) := \int e^{-i\hinv\theta} \frakh(x,\theta)\, d\theta
\end{equation}
is a smooth function of $x$ for each $\h$, 
since the wave-front set of $\frakh$ does not contain covectors conormal to 
the fibers of the projection $U\times\bbR\to U$. 

\begin{definition}
Let $r$ be a half integer and $\wt\Sigma\subset T^*(U\times\bbR)_+$
a conic isotropic submanifold.  
We then define $\tilde I^r(\wt\Sigma)$ to be the class of $\h$-dependent
functions on $U$
given by the partial Fourier transform (\ref{deBorelization}), as $\frakh$ ranges over
the space $I^{-r+\frac n2}(U\times\bbR, \wt\Sigma)$ of 
Hermite distributions in the sense of Boutet de Monvel, \cite{BG}.
\end{definition}

The alternative approach of this section is embodied by the following
\begin{theorem}\label{AlternativeApproach}
Let $\Sigma\subset T^*U$ a connected isotropic submanifold (not necessarily conic),
and assume that there is  a conic isotropic submanifold,
$\wt\Sigma\subset T^*(U\times\bbR)_+$, such that
\[
\Sigma = \{(x,p)\in T^*U\;;\; \exists\theta\in\bbR\ (x,\theta\;;\; p, \kappa=1)\in \wt\Sigma\}.
\]
Then there exists $\theta_0\in\bbR$ such that
\[
\tilde I^r(\wt\Sigma)= e^{i\hinv\theta_0}I^r(\Sigma).
\]
\end{theorem}

\begin{remark}
The hypothesis of the theorem can be seen to be equivalent to the existence
of a function $\psi:\Sigma\to\bbR$ such that 
\begin{equation}\label{}
d\psi = \iota^*(pdx),
\end{equation}
where $\iota:\Sigma\hookrightarrow T^*U$ is the inclusion and $pdx$ the
tautological one-form.  It therefore is always satisfied (micro) locally.
To such a function one associates the isotropic $\wt\Sigma\subset T^*(U\times\bbR)$ given by
\[
\wt\Sigma = \{(x,\theta\;;\;\xi,\kappa)\;;\; \theta = \psi(x,p),\ \xi = \kappa p,\, (x,p)\in\Sigma\}.
\]
The overall phase factor $e^{i\hinv\theta_0}$ reflects the fact that the
function $\psi$ is only defined up to a constant, if $\Sigma$ is connected.
\end{remark}

\begin{remark}
To make the statement in the theorem clear, the conclusion is that there is a $\theta_0$ such 
that for any Hermite distribution 
$\frakh\in I^{-r+\frac n2}(U\times\bbR, \wt\Sigma)$
there is a corresponding $\Upsilon\in I^r(\Sigma)$ such that 
$\widehat\frakh = e^{i\hinv\theta_0}\Upsilon$.
Moreover, there is a simple
correspondence between the symbols of $\frakh$ and $\Upsilon$,
which roughly speaking says that under the identification of $\Sigma$ 
with the subset $\wt\Sigma\cap\{\kappa=1\}$, the symbol of 
$\Upsilon$ is the restriction of the symbol of $\frakh$ to the set 
$\kappa = 1$
(the symplectic normal spaces of $\Sigma$ and of $\wt\Sigma$
are naturally isomorphic at corresponding points).

\end{remark}

\medskip
The proof of Theorem \ref{AlternativeApproach} will take the remainder of this section.  

\begin{proposition}
If $\Sigma_0$ is the model isotropic, then for a suitable choice of $\wt\Sigma_0$ one has
$\tilde I^r(\wt\Sigma_0)= I^r(\Sigma_0)$.
\end{proposition}
\begin{proof}
Recall that the model isotropic $\Sigma_0\subset T^*\bbR^{k+l}$ is
\[
\Sigma_0 = \{(t, u=0; \xi=0)\}
\]
where the coordinates on $T^*\bbR^{k+l}$ are $(t, u\;;\; \xi = (\tau, \mu))$.
A canonical ``lift" to $T^*(\bbR^{k+l}\times\bbR)$ is 
\[
\wt\Sigma_0 = \{(t, u=0, \theta = 0; \xi=0, \kappa>0)\}.
\]
To simplify notation, we'll take without loss of generality $r=0$.

\smallskip
To show that $I^0(\Sigma_0)\subset \tilde{I}^0(\wt\Sigma_0)$,
let $\Upsilon(t,u,\h) = a(t, \h^{-1/2}u, \h)\in I^0(\Sigma_0)$. 
Recall that $a(t,u,\h)\sim \sum_{j=0}^\infty a_j(t,u)\,  \h^{j/2}$, where the $a_j$
are Schwartz in the variables $u$.  It will be enough to show that 
$\Upsilon_j\in \tilde{I}^j(\wt\Sigma_0)$, where
\[
\Upsilon_j(t,u,\h) = a(t, \h^{-1/2}u)\h^{j/2}.
\]
Let
\[
\frakh_j(t,u,\theta) = \int_0^\infty e^{i\kappa\theta} a_j(t, \kappa^{1/2}u)\, \kappa^{-j/2}\, d\kappa.
\]
Let $\hat{a}_j$ be the Fourier transform of $a_j$ in the $u$ variables, so that
\[
a_j(t,u) = \int e^{i\zeta\cdot u}\, \hat{a}_j(t,\zeta)\, d\zeta.
\]
Then
\[
\frakh_j(t,u,\theta) = \int_{\kappa >0} e^{i(\kappa\theta + \kappa^{1/2}\zeta\cdot u)}\,
\hat{a}_j(t,\zeta)\,\kappa^{-j/2}\, d\kappa\,d\zeta.
\]
If we let $\eta = \sqrt{\kappa}\,\zeta$, by substitution we get
\[
\frakh_j(t,u,\theta) = \int_{\kappa >0} e^{i(\kappa\theta + \eta\cdot u)}\,\hat{a}_j(t,\eta/\sqrt{\kappa})\,
\kappa^{-l/2}\, d\eta\,d\kappa.
\]
But this expression shows exactly that $\frakh_j$ is an Hermite distribution in the stated class.

\bigskip

Now we prove that $\tilde{I}^0(\wt\Sigma_0)\subset I^0(\Sigma_0)$.  
Let $\frakh$ be a Hermite distribution associated with 
$\wt\Sigma_0$.
By the discussion on symbols, $\frakh$ induces a symbol in the same symbol space as
those of elements in $I^0(\Sigma_0)$.  Let $\Upsilon_0\in I^0(\Sigma_0)$ be any
Hermite state with the same symbol as the one induced by $\frakh$.  Then, by the
previous part of the proof, $\Upsilon_0\in \tilde{I}^0(\wt\Sigma_0)$
and $\widehat\frakh-\Upsilon_0\in\tilde{I}^{1/2}(\wt\Sigma_0)$.  
Repeat the argument inductively, to obtain $\Upsilon_\infty\in \tilde{I}^0(\wt\Sigma_0)$ such that 
$\widehat\frakh-\Upsilon_\infty\in\tilde{I}^{\infty}(\wt\Sigma_0)$.
\end{proof}

\begin{proposition}
The classes $\tilde I^r(\wt\Sigma)$ are equivariant under the action of semiclassical
FIOs on $C^\infty(U)$.
\end{proposition}
\begin{proof}
By \cite{PU2}, semiclassical FIOs  on $C^\infty(U)$
correspond to H\"ormander's FIOS on $C^\infty(U\times\bbR)$ that commute
with the $\bbR$ action on $C^\infty(U\times\bbR)$, and the classes of 
Hermite distributions are invariant under FIOs, see \cite{BG} \S 3.
\end{proof}

\subsubsection{Global theory}

\medskip
Let $M$ be a manifold, and $\Sigma\subset T^*M$ a isotropic
submanifold.
Let $\iota:\Sigma\hookrightarrow T^*M$ be the inclusion, and denote by
$pdx$ the tautological one-form of $T^*M$.  The isotropic condition
on $\Sigma$ is that $\iota^*pdx$ is closed.  It is not generally true that 
$\iota^*pdx$ is exact, as assumed in the previous section, but 
in some cases it is true that $\Sigma$ satisfies 
the {\em Bohr-Sommerfeld condition}, namely, 
 that
$\exists\ f:\Sigma\to S^1$ smooth such that
\begin{equation}\label{}
 \iota^*pdx = \sqrt{-1}\, d\log(f).
\end{equation}
We can then lift  $\Sigma$ to a homogeneous submanifold
of the cotangent bundle of $M\times S^1$:
\begin{equation}\label{}
\widetilde\Sigma := \left\{ (x, f(x,p)
\;;\; \kappa p,\kappa)\in T^*(M\times S^1)
\;;\; (x,p)\in \Sigma,\, \kappa\in\bbR^+\right\}.
\end{equation}
It is not difficult to check that $\widetilde\Sigma$ is an
isotropic submanifold of $T^*(M\times S^1)$.  Conversely,
a conic isotropic submanifold $\widetilde\Sigma\subset T^*(M\times S^1)_+$,
where 
\[
T^*(M\times S^1)_+ := \{(t, u, \theta\;;\;\tau,\eta,\kappa)\in T^*(M\times S^1)\;;\; \kappa>0\},
\]
gives rise to an isotropic $\Sigma\subset T^*M$ by the process of reduction:
\[
\Sigma = \left(\widetilde\Sigma\cap\{\kappa=1\}\right)/S^1.
\]

The results of the previous section imply:

\begin{theorem}
Let $\Upsilon\in I(M\times S^1, \widetilde{\Sigma})$ be an Hermite
distribution in the sense of Boutet de Monvel,
\cite{BG}, and let
\[
\Upsilon(x,\theta) = \sum_m \Upsilon_m(x)\,e^{im\theta}
\]
be its Fourier series.  Then the family $\{\Upsilon_m\}$
is an isotropic state associated to $\Sigma$,
in the sense of Definition \ref{generaldef},
after the substitution $\h = 1/m$.  
\end{theorem}

\section{Applications}

In this section we will briefly describe some applications of the theorems above. 
\subsection{{Propagation of coherent states}}

We begin with:
\begin{definition}
Let $X$ be a manifold, and $p_0 =(x_0,\xi_0)\in T^*X$.  By a {\em coherent 
state} centered at $p_0$ we will mean any element $\Upsilon\in I^0(\{p_0\})$, 
that is, any isotropic state associated to $\Sigma = \{p_0\}$.
\end{definition}

Now let $P$ be a self-adjoint
semiclassical pseudo-differential operator of order zero on $X$.
As an important example, if $X$ is a Riemannian manifold and $V: X \to \mathbb R$ a 
$C^\infty$ function,  we can take for $P$ the Schr\"odinger operator 
\begin{equation}\label{SchOp}
P (\psi)= \hbar^2 \Delta \psi + V\psi
\end{equation}
with $\Delta$ the Laplace-Beltrami operator.
Let us denote by $H(x,\xi):T^*X\to\bbR$ the symbol of $P$ (so that, in the example
$H(x,\xi) = \norm{\xi}^2+v(x)$), and let us assume that $H$ is proper.
For each function $\psi_0 \in C^\infty(X)$, 
let $\psi(x, t)$ be the solution of the Schr\"odinger
equation
\begin{equation}\label{SchEq}
i\hbar \frac{\partial \psi}{\partial t} = P(\psi)
\end{equation}
with initial condition $\psi(x, 0)=\psi_0$. Then, for each $t$, the map $\psi_0 \mapsto 
\psi(x, t)$ is a semiclassical Fourier integral operator 
\begin{equation}\label{4-13}
F_t: C^\infty(X) \to C^\infty(X )
\end{equation}
associated to the graph of the time $t$ map 
\[
\phi_t: T^*X\to T^*X
\]
of the Hamilton flow of $H$.  

\medskip
The first result on propagation of coherent states is the following:
\begin{theorem}
Let $\Upsilon$ be a coherent state centered at $p_0\in T^*X$.  Then, for each $t\in\bbR$,
$F_t(\Upsilon)$ is a coherent state, namely
$F_t(\Upsilon)\in I^0(\{\phi_t(p_0)\})$.  Moreover, the symbol of $F_t(\Upsilon)$
is the result of applying Mp$(g)$ to the symbol of $\Upsilon$, where
\[
g = d(\phi_t)_{p_0}: T_{p_0}T^*X \to T_{\phi_t(p_0)}T^*X.
\]
\end{theorem}
The proof is immediate, by the global theory of $\S 3.2$, in particular the Definition
\ref{GlobalSymbolDef} of the symbol on manifolds.  (Note that the symbols of coherent
states do not have a half-form component, since the isotropic is just a point.)

\bigskip
We now show that if $F:C^\infty(X)\to C^\infty(X\times\bbR)$ is the operator
\[
F(\psi_0) = \psi(x,t)
\]
and $\Upsilon$ is a coherent state, then $F(\Upsilon)$ is still an isotropic state on $X\times\bbR$.
This is {\em not} immediate from the results of \S 3, because the operator $F$ is a semi-classical
FIO associated to a canonical relation that is not a transformation.

\medskip
In fact we'll prove something slightly more general.
Let $X$ and $Y$ be manifolds, and let $\Gamma \subset T^*X  \times T^*Y$ be a canonical relation
(not necessarily the graph of a  symplectomorphism; in particular, we do not 
assume that $X$ and $Y$ have the same dimension).
Let $F: C^\infty(Y) \to C^\infty(X)$ be a semiclassical Fourier integral operator quantizing $\Gamma$. We will prove 
\begin{theorem}\label{thm4.1}
Let $p_0=(y_0, \eta_0) \in T^*Y$ be a regular value of the projection $\pi: \Gamma \to T^*Y$. Then $\pi^{-1}(p_0)=\Sigma$ is an isotropic submanifold of $T^*X$. Moreover, if $\psi_\hbar \in C^\infty(Y)$ is a coherent state centered at $p_0$, then 
\begin{equation}
F(\psi_\hbar) \in I(\Sigma).
\end{equation}
\end{theorem}
\begin{proof}
By a partition of unity argument we can assume that the Schwartz kernel of $F$ is supported on an open set $U \times V$, where $U$ and $V$ are coordinate patches in $Y$ and $X$, and hence we can assume without loss of generality that $Y=\mathbb R^m$ and $X=\mathbb R^n$. Less obviously we can also assume that $\Gamma$ is a horizontal submanifold of $T^*(X \times Y)$, i.e. that its projection onto $X \times Y$ is a bijection.  To see this we note:
\begin{lemma}
There exist linear symplectomorphisms, $A: T^*\mathbb R^n \to T^*\mathbb R^n$ and $B: T^*\mathbb R^m \to T^* \bbR^m$, such that $A F B$ is horizontal. Hence ``Theorem \ref{thm4.1} for $AFB$" implies Theorem \ref{thm4.1} for $F$".
\end{lemma}
(We will omit the proof of this since it is, more or less verbatim, the same as the proof of theorem \ref{MLT1} in the main lemma segment of \S 2.)

Thus we are reduced to proving theorem \ref{thm4.1} for F.I.O.'s of the form
\begin{equation}\label{horFIO}
F u(x) =\int a(x, y, \hbar) e^{\frac{i\phi(x, y)}{\hbar} }u(y)dy
\end{equation}
where $\phi$ is the defining function of $\Gamma$, i.e. 
\begin{equation}\label{phiDefGamma}
(x, \xi, y, \eta) \in \Gamma \Leftrightarrow \xi=\frac{\partial \phi}{\partial x} \mbox{\ and\ } \eta=-\frac{\partial \phi}{\partial y}.
\end{equation}

We can also assume without loss of generality that $p_0=(y_0, \eta_0)=(0, 0)$. Hence the transversality condition in theorem \ref{thm4.1} asserts that the equations 
\begin{equation}\label{4-5}
\frac{\partial \phi}{\partial y}(0, 0)=0,\quad \xi=\frac{\partial \phi}{\partial x}(x, 0)
\end{equation}
are a non-degenerate system of defining equations for $\Sigma$. In particular 
\begin{equation}\label{4-6}
d\frac{\partial \phi}{\partial y_i}(x, 0), \quad i=1, \cdots, m
\end{equation}
are linearly independent, and hence by a change of coordinates, we can assume 
\begin{equation}\label{4-7}
\frac{\partial \phi}{\partial y_i}(x, 0)=x_i, \quad i=1, \cdots, m
\end{equation}
and 
\begin{equation}\label{4-8}
\phi(x, y) = \phi_0(x)+ \sum_{i=1}^m x_i y_i + \sum a_{i,j}(x, y)y_i y_j.
\end{equation}
Thus if $\psi_{p_0}$ is the coherent state 
\begin{equation}\label{4-9}
\psi(\frac{y}{\hbar^{1/2}}), \quad \psi \in \mathcal S(\bbR^m),
\end{equation}
$F\psi_{p_0}$ is equal to 
\begin{equation}\label{4-10}
e^{\frac{i\phi_0(x)}{\hbar}} \tilde \psi(t, \frac{u}{\hbar^{1/2}}, \hbar),
\end{equation}
where $t=(x_{m+1}, \cdots, x_n)$, $u=(x_1, \cdots, x_m)$ and 
\begin{equation}\label{4-11}
\tilde \psi(t. u, \hbar) = \hbar^{n/2} \int e^{iu\cdot \mu} f_a(t, u, \mu, \hbar) d\mu
\end{equation}
with $f_a(t, u, y, \hbar)$ given by 
\begin{equation}\label{4-12}
a(t, \hbar^{1/2}u, \hbar^{1/2}\mu, \hbar)e^{i\sum a_{jk}(t,u, \hbar^{1/2}\mu)\mu_i \mu_j} \psi(\mu).
\end{equation}
Thus in particular, by (\ref{4-11}), $\tilde \psi(t, u, \hbar)$ is rapidly decreasing as a function of $u$ and hence (\ref{4-10}) is an isotropic state with microsupport on the set, $u=0$ and $\xi=\frac{\partial \phi_0}{\partial x}$. Note however that by (\ref{4-5}) this set is just the isotropic subset, $\Sigma=\pi^{-1}(p_0)$ of $T^*X$. 
\end{proof}

\begin{corollary}
Let $P$ be a semi-classical self-adjoint pseuodifferential operator on $X$ with proper symbol,
and
\begin{equation}\label{4-13}
F: C^\infty(X) \to C^\infty(X \times \bbR)
\end{equation}
the fundamental
solution of the Schr\"odinger equation (\ref{SchEq}).  Let $\Upsilon$ be a coherent state
centered at $p_0=(x_0,\xi_0)\in T^*X$.  Then $F(\Upsilon)\in I^0(\Sigma)$, where
\begin{equation}\label{4-17}
\Sigma = \{(x, \xi; t, \tau), \quad (x, \xi)=\phi_t(x_0, \xi_0), \, \tau=p(x_0, \xi_0)\},
\end{equation}
\end{corollary}
\begin{proof}
This follows from the previous theorem and the fact that the 
canonical relation, $\Gamma$, of the operator $F$ is defined by the condition:
\begin{equation}\label{4-14}
((x, \xi), (y, \eta), (t, \tau))  \in \Gamma
\end{equation}
if and only if
\begin{equation}\label{4-15}
(y, \eta) = \phi_t(x, \xi)\quad\text{and}\quad
\tau=p(x, \xi)=p(y, \eta). 
\end{equation}
\end{proof}

\medskip

\begin{remark}  Keeping the notation of the previous corollary,
suppose that the trajectory, $\gamma\subset T^*X$, of $p_0$ is periodic of period $T>0$.
Let $\rho\in C^\infty(\bbR)$ be $T$-periodic, and consider the push-forward
\[
u = \int_0^T F(\Upsilon)\, \rho(t)\, dt,
\]
where $\Upsilon$ is a coherent state centered at $p_0$.
We claim that one can show that $u\in I^{1/2}(\gamma)$.  Moreover,
if $\lambda =H(p_0)$, then $(P-\lambda)u\in I^{1}(\gamma)$, because the symbol of
$P-\lambda$ is zero on $\gamma$ (see Theorem \ref{symbolcalc0}).
But note that, since the Hamilton field of $H$ is tangent to $\gamma$,
by the first transport equation (Theorem \ref{symbolcalc1}) we in fact have that
$(P-\lambda)u\in I^{3/2}(\gamma)$.

In suitable situations one can 
construct quasi-modes $u\in I(\gamma)$ by symbolic methods, that is, non-trivial
isotropic states satisfying
$(P-\lambda)u\in I^r(\gamma)$ for all $r>0$.  By the discussion above, the first 
obstruction is that the symbol of $u$ should satisfy the characteristic equation
(equal zero) at some point on $\gamma$.
Then, by invariance with respect 
to the bicharacteristic flow, the symbol of $u$ satisfies the second transport equation 
at all points of this bicharacteristic
\end{remark}

\color{black}
\subsection{A result on the pseudospectrum}

The following theorem has a symbolic proof, and as an immediate corollary 
we obtain a result on the pseudospectrum of a non self-adjoint pseudodifferential operator.
The latter result was first proved in \cite{DSZ}, by other methods.
\begin{theorem}
Let $A$ be a semiclassical pseudodifferential operator on a manifold $M$ with
principal symbol $H: T^*M\to\bbC$.  Let $p\in T^*M$ be such that
\begin{equation}\label{pbassumption}
\PB{\Re(H)}{\Im(H)}(p) <0.
\end{equation}
Then there exists $\Upsilon\in I^{0}(\{p\})$ with non-zero symbol such that
\begin{equation}\label{}
\left( A-\lambda I\right)(\Upsilon) = O(\hbar^\infty),
\end{equation}
where $\lambda = H(p)$ and the asymptotics are in the $C^\infty$ topology.
\end{theorem}
\begin{proof}
If we let $P= A-\lambda I$, then for any $\Upsilon\in I^{0}(\{p\})$
 $P(\Upsilon)\in I^{1/2}(\{p\})$, because the symbol of $P$ 
vanishes at $p$.  By Theorem \ref{symbolcalc1}, the symbol of $P(\Upsilon)$ is
\begin{equation}\label{symbolOne}
\sigma_{P(\Upsilon)}=d\rho(\xi)(\sigma_\Upsilon),
\end{equation}
where $\xi\in T_p(T^*M)$ is the Hamilton field of $H$ evaluated at $p$, and $d\rho$ is the
infinitesimal Schr\"odinger representation of the Heisenberg group of $V:=T_p(T^*M)$.

\smallskip\noindent
{\sc Claim:}  {\em  Under the assumption (\ref{pbassumption}), the
 operator $d\rho(\xi): \calS(V)\to\calS(V)$ is onto and has a non-trivial
kernel.  Here $\calS(V)$ is the space of smooth vectors for the metaplectic representation
of $V$ (Schwartz functions).}

This claim is precisely Lemma 3.1 of \cite{BoU}, but we sketch the simple argument:  
It suffices to prove the statement in a model of the metaplectic representation, say
$\calS(V) =   \calS(\bbR^n)$ with the usual Schr\"odinger representation, after 
choosing a symplectic basis $(e_1,\ldots, e_n, f_1,\ldots,f_n)$ on $V\cong\bbR^{2n}$ so that
$\xi =\epsilon e_1 +i f_1$.  The condition on the sign of the Poisson bracket corresponds to:
$\epsilon >0$, and in this model
\begin{equation}\label{}
d\rho(\xi) = \calL = \frac{\partial\ }{\partial x_1} + \epsilon x_1.
\end{equation}
The kernel of this operator contains Schwartz functions (e.g. 
$e^{-(\epsilon x_1^2 + x_2^2+\cdots +x_n^2)/2}$), and by variation of parameters 
one can check that the solution of the
ODE $\calL u = f$, where $f$ is Schwartz, is Schwartz.

\medskip
With this claim at hand, choose the symbol of $\Upsilon$ so that (\ref{symbolOne}) is zero, and
denote by $\sigma_1$ the symbol of $P(\Upsilon)\in I^{1}(\{p\})$.

Next, let us look for $\gamma_1\in I^{1/2}(\{p\})$ so that
\begin{equation}\label{symbolTwo}
P(\Upsilon +\gamma_1)\in I^{3/2}(\{p\}).
\end{equation}
For any $\gamma_1\in I^{1/2}(\{p\})$ one has
$P(\gamma_1)\in I^{1}(\{p\})$ and
(\ref{symbolTwo}) will hold if $\gamma_1$ satisfies
$
\sigma_{P(\gamma_1)} = -\sigma_1,
$
which, once again by the first transport equation, amounts to
\begin{equation}\label{etc}
d\rho(\xi)(\sigma_{\gamma_1}) = -\sigma_1.
\end{equation}
By the previous Claim there is a solution to this problem, and we take $\gamma_1$ to have
it as its symbol.  This constructs $\Upsilon_1 = \Upsilon +\gamma_1$ such that
$P(\Upsilon_1)\in I^{3/2}(\{p\})$.

Now continue this process to all orders, at each step solving an inhomogeneous equation
of the form (\ref{etc}).
\end{proof}

In the situation of the previous Theorem, one can conclude that $\lambda$ is in 
{\em the semiclassical pseudospectrum of } $A$, since
\begin{equation}\label{}
\lim_{\h\to 0} \frac{\norm{(A-\lambda I)(\Upsilon)}}{\norm{\Upsilon}} = 0.
\end{equation}
This result on the pseudospectrum of $A$ 
was previously proved by Dencker, Sj\"ostrand and Zworski in \cite{DSZ}.

\subsection{Complex analytic examples of isotropic states}

Here we briefly indicate how to construct many examples of isotropic
states in the complex analytic category.

Consider a compact complex manifold $Z$, a holomorphic line bundle, 
$\mathbb L \to Z$, and a Hermitian inner product, $\langle \ , \ \rangle$, on 
$\mathbb L$ which is positive definite in the sense that the curvature form 
associated with the intrinsic metric connection on $\mathbb L$ is a K\"ahler 
form. Let $\mathbb L^*$ be the dual line bundle to $\mathbb L$. Then 
\[  
D(\mathbb L^*) = \{(z, v) \in \mathbb L^*, \langle v, v \rangle_z^* <1 \}
\]
is a strictly pseudoconvex domain.
Let $X = \partial D$ equipped with the volume form $\alpha \wedge 
(d\alpha)^{n-1}$, where $\alpha$ is the connection form.
$X$ is a circle bundle over $Z$.  We let
\begin{equation}\label{SzegoPro}
\Pi: L^2(X) \to H^2(X)
\end{equation}
be the Szeg\"o projector.  (Here $H^2(X)$ is the space of boundary values
of holomorphic functions on $D$.)  This projector was studied in,
for instance, \cite{BG} and \cite{BS}.  It is known that the
Schwartz kernel of $\Pi$ is an Hermite distribution associated to the 
conic isotropic
\begin{equation}\label{szegoIsotropic}
\{(x,x\;;\; r\alpha_x, -r\alpha_x)\in T^*X\times T^*X\;;\; x\in X,\ r>0\}.
\end{equation}

Now let $u$ be a Hermite (or lagrangian) distribution on $X$.  
If the isotropic submanifold of $u$ satisfies a ``clean intersection condition" with 
respect to (\ref{szegoIsotropic}), then 
$\Pi(u)$ is a Hermite distribution on $X$,  with respect to
a conic isotropic
\[ 
\wt\Sigma\subset \{(x, r\alpha_x)\;;\; x\in X,\ r>0\}.
\]
(For details on this composition theorem, see \S 7 in \cite{BG}.)
Furthermore, $\Pi(u)$ is 
in the generalized Hardy space, and we can decompose it as
\[
\Pi(u) = \sum_{m=1}^\infty u_m,
\]
with respect to the circle bundle action on $X$.
Specifically, for each $m$, $u_m\in \mathcal H_m$, 
where  $\mathcal H_m$ is the space of functions in $H^2(X)$ which transform under the action of $S^1$ by the character $e^{im\theta}$. 
(Note, by the way, that for each $m$, $u_m$ can be
interpreted as a holomorphic section of $\mathbb L^m\to Z$.) 

\medskip
The results of \S 3.5.2 immediately imply:
\begin{corollary}
Let $U\subset Z$ be an open set
such that the bundle $\pi: X\to Z$ is trivial over $U$,
and fix a trivialization $\pi^{-1}(U)\cong U\times S^1$.  In this trivialization,
let
\[
u_m(z,\theta) = \Upsilon_m(z)\, e^{im\theta}.
\]
Then the sequence $\{\Upsilon_m\}$ is an isotropic semiclassical state
on $U$, where $\h = 1/m$.
\end{corollary}

Borthwick, Paul and Uribe considered in \cite{BPU}  the case when $u$
is Lagrangian and gave some applications.
We hope to provide details and applications of the case when $u$ is Hermite 
in a future paper.

%

\end{document}